\documentclass[10pt,reqno]{amsart}
\usepackage{amssymb,mathrsfs,graphicx,amsmath,mathtools}
\usepackage{ifthen}
\usepackage[margin=1in]{geometry}
\usepackage{caption}
\usepackage{sidecap}
\usepackage{rotating}
\usepackage{xcolor}
\usepackage{tikz}
\usetikzlibrary{arrows.meta}
\usetikzlibrary{calc,decorations.pathreplacing,intersections}

\usepackage{hyperref}

\provideboolean{shownotes} 
\setboolean{shownotes}{true} 
\newcommand{\margnote}[1]{
\ifthenelse{\boolean{shownotes}}%
{\marginpar{\raggedright\tiny\texttt{#1}}}%
{}%
}
\newcommand{\hole}[1]{
\ifthenelse{\boolean{shownotes}}%
{\begin{center} \fbox{ \rule {.25cm}{0cm} \rule[-.1cm]{0cm}{.4cm}
\parbox{.85\textwidth}{\begin{center} \texttt{#1}\end{center}} \rule
{.25cm}{0cm}}\end{center}} {} }

\title[Large-time behavior for the Euler--alignment system with nonlocal forces]{Exponential and algebraic decay in  Euler--alignment system with nonlocal interaction forces}

\author[Carrillo]{Jos\'{e} A. Carrillo}
\address[Jos\'{e} A. Carrillo]{\newline Mathematical Institute, University of Oxford, 
    \newline Oxford OX2 6GG, United Kingdom}
\email{jose.carrillo@maths.ox.ac.uk}

\author[Choi]{Young-Pil Choi}
\address[Young-Pil Choi]{\newline Department of Mathematics, Yonsei University, \newline
Seoul 03722, Republic of Korea}
\email{ypchoi@yonsei.ac.kr}

\author[Koo]{Dowan Koo}
\address[Dowan Koo]{\newline Mathematical Institute, University of Oxford, 
    \newline Oxford OX2 6GG, United Kingdom}
\email{dowan.koo@maths.ox.ac.uk}

\author[Tse]{Oliver Tse}
\address[Oliver Tse]{\newline Department of Mathematics and Computer Science, Eindhoven University of Technology, 
\newline P.O. Box 513, 5600 MB Eindhoven, The Netherlands}
\email{o.t.c.tse@tue.nl}

\numberwithin{equation}{section}

\newtheorem{theorem}{Theorem}[section]
\newtheorem{lemma}{Lemma}[section]
\newtheorem{corollary}{Corollary}[section]
\newtheorem{proposition}{Proposition}[section]
\newtheorem{remark}{Remark}[section]

\newcommand{\R}{\mathbb R}

\newcommand{\ls}{\lesssim}

\newcommand{\bq}{\begin{equation}}
\newcommand{\eq}{\end{equation}}
\newcommand{\e}{\varepsilon}
\newcommand{\lt}{\left}
\newcommand{\rt}{\right}
\newcommand{\lal}{\langle}
\newcommand{\ral}{\rangle}
\newcommand{\pa}{\partial}

\newcommand{\calL}{\mathcal{L}}

\DeclareMathOperator*{\esssup}{ess\,sup}
\DeclareMathOperator*{\supp}{supp}

\newcommand{\scrB}{\mathscr{B}}
\DeclareMathOperator*{\argmin}{arg\,min}

\begin{document}
\allowdisplaybreaks


\subjclass[2020]{}

\begin{abstract} 
We investigate the large-time behavior of the pressureless Euler system with nonlocal velocity alignment and interaction forces, with the aim of characterizing the asymptotic convergence of classical solutions under general interaction potentials $W$ and communication weights. We establish quantitative convergence in three settings. In one dimension with $(\lambda,\Lambda)$-convex potentials, i.e., potentials satisfying uniform lower and upper quadratic bounds, bounded communication weights yield exponential decay, while weakly singular ones lead to sharp algebraic rates. For the Coulomb--quadratic potential $W(x)=-|x|+\frac12 |x|^2$, we prove exponential convergence for bounded communication weights and algebraic upper bounds for singular communication weights. In a multi-dimensional setting with uniformly $(\lambda,\Lambda)$-convex potentials, we show exponential decay for bounded weights and improved algebraic decay for singular ones. In all cases, the density converges (up to translation) to the minimizer of the interaction energy, while the velocity aligns to a uniform constant. A unifying feature is that the convergence rate depends only on the local behavior of communication weights: bounded kernels yield exponential convergence, while weakly singular ones produce algebraic rates. Our results thus provide a comprehensive description of the asymptotic behavior of Euler--alignment dynamics with general interaction potentials.
\end{abstract}

\maketitle \centerline{\date}


%
%
%
%
\section{Introduction}
We are concerned with the dynamics of a pressureless Euler system with nonlocal forces, which couples interaction through a potential with velocity alignment effects. The system reads
\begin{align}\label{main_sys}
\begin{aligned}
&\pa_t \rho + \nabla \cdot (\rho u) =0, \quad (t,x) \in \R_+ \times \R^d,\cr
&\pa_t (\rho u) + \nabla \cdot(\rho u \otimes u) =  - \rho \nabla W \star \rho -\rho \int_{\R^d} \phi(|x-y|)(u(x) - u(y))\rho(y)\,dy.
\end{aligned}
\end{align}
Here $\rho(t,x)\ge 0$ denotes the density and $u(t,x)\in\R^d$ the velocity field with dimension $d\geq 1$. The interaction potential $W:\R^d\to\R$ models distance-dependent interactions among individuals, and is assumed to be symmetric, $W(x)=W(-x)$, in accordance with Newton's action-reaction principle. The second term represents a weighted velocity alignment, where the tendency of agents to match their velocities is modulated by the communication weight $\phi:[0,\infty)\to[0,\infty)$, which depends only on pairwise distance and is nonincreasing so that closer agents exert stronger alignment forces.

The system \eqref{main_sys}, commonly referred to as the {\it Euler--alignment model}, originates as a hydrodynamic limit of agent-based alignment dynamics such as the Cucker--Smale model \cite{CS07}, where each particle adapts its velocity to its neighbors through nonlocal interactions. The additional convolution term with $\nabla W$ accounts for attractive-repulsive forces induced by a potential $W$, thus incorporating both alignment and aggregation effects into the macroscopic description. Such equations provide a continuum framework for collective phenomena across various disciplines: in biology, they describe large-scale behaviors such as the flocking of birds, the schooling of fish, or the coordinated motion of insects; in physics, they capture the self-organization of active matter and the emergence of coherent patterns in particle ensembles; and in the social sciences, they serve as models for opinion dynamics, consensus formation, and other collective behaviors in human society.
We refer to \cite{CFTV10,CCP17, CHL17, MMPZ19, Shv21, Shv24} for general introductions, surveys, and further results on Cucker--Smale type flocking models, as well as to the references therein for broader perspectives.

From the mathematical point of view, the Euler--alignment system presents several challenging features. On the one hand, the absence of pressure renders the dynamics susceptible to concentration and potential blow-up. On the other hand, the nonlocal alignment mechanism has a regularizing effect, promoting velocity alignment and, under suitable conditions, preventing singularity formation. A central problem is thus to understand the competition between these two mechanisms and to establish rigorous results on global well-posedness, asymptotic behavior, and the emergence of collective patterns.  

In recent years, substantial progress has been made on Euler-type flocking models.  For the system \eqref{main_sys} without interaction potential $W$, global existence of solutions and their long-time behavior have been studied in different regimes \cite{HT08, TT14, HKK15, ST17, ST18, KT18, DKRT18, C19, CH19,  LS19, CJu24, DMPW19, Shv21, LTST22,   Shv24, CaGa, Tpre}, while the corresponding kinetic models have been analyzed in \cite{COP09, HL09, CFRT10, KMT13, CCH14, HKK14,  MMPZ19, CZ21}. Critical threshold phenomena, conditions on initial data separating global regularity from finite-time blow-up, have been identified for various alignment kernels \cite{TT14,  CCTT16, T20, BLT23, LTWpre}.  Connections with many-particle systems and kinetic descriptions have also been developed, yielding mean-field limits \cite{CC21, FP24, ACPSp} or hydrodynamic limits \cite{KMT15, FK19, CCJ21, BT25}. We also mention related work on Lagrangian trajectories \cite{Les20} and sticky dynamics for the one-dimensional Euler--alignment model \cite{ LeTa23, Ga25}.  

In the presence of the interaction potentials or other external forces, the global existence and large-time behavior of solutions has likewise been discussed in \cite{CCP17, CCT19, CWKZ19, CWKZ20,  ST2020a, ST2021b, Shv21, CCTZ25, LTWpre}. For the quadratic potential $W(x) = \frac{|x|^2}2$, we find
\[
\nabla W \star \rho = \int_{\R^d} (x-y)\rho(y)\,dy = x \|\rho\|_{L^1} - \rho_c, \quad \rho_c := \int_{\R^d} x \rho(x)\,dx,
\]
so the interaction force reduces, after recentering at the center of mass, to an external quadratic field, see the reformulation in \cite[Remark 2.2]{ST2020a}. Moreover, since momentum is conserved, when $\phi\equiv1$ the alignment term becomes a linear damping toward the average velocity:
\[
\int_{\R^d} (u(y)-u(x)) \rho(y)\,dy = \int_{\R^d} \rho u\,dx -u(x).
\]
In one dimension, the system \eqref{main_sys} with $\phi \equiv 1$ and Coulomb--quadratic potential $W(x)=-|x|+\frac12|x|^2$ can then be identified with a damped pressureless Euler--Poisson system with constant background states; see e.g. \cite{ELT01,  CCZ16, BL20, CKKT} for sharp critical threshold results.  

In the current work, we quantify the decay of classical solutions to \eqref{main_sys}. The classical notion of {\it flocking} combines uniform boundedness of the density support with velocity alignment:
\[
\sup_{t \ge 0} {\rm diam}\big({\rm supp}(\rho(t))\big) < \infty, \quad \sup_{x,y \in {\rm supp}(\rho(t))} |u(t,x) - u(t,y)| \to 0 \quad \mbox{as } t \to \infty.
\]

In the case without interaction potential ($W\equiv 0$), flocking requires sufficiently strong long-range communication. More precisely, {\it unconditional flocking}, that is, flocking emerging for all initial data, occurs if the communication weight has a {\it heavy tail},
\[
\int^\infty \phi(r)\,dr = \infty,
\]
as shown in \cite{ HT08,HL09,CFRT10}. This condition ensures that agents can still influence each other even at arbitrarily large distances, so that alignment eventually prevails regardless of the initial configuration. On the other hand, in the presence of quadratic confinement, it was proved in  \cite{ST2020a} that the long-range requirement for unconditional flocking can be weakened to
\bq\label{thin}
\int^\infty r\phi(r)\,dr = \infty,
\eq
which admits {\it thinner tails}. Here, the confining potential effectively counteracts dispersion of the support, keeping the population within a bounded region and thereby amplifying the influence of $\phi$. Thus, confinement reduces the burden on the communication weight: even kernels whose decay is faster than that required by the heavy-tail condition may still lead to flocking. Nevertheless, both criteria impose strong restrictions on the decay of $\phi$ at infinity. In particular, when $\phi$ is integrable, flocking may still occur, but only under additional assumptions on the initial configuration.
The precise asymptotics in the presence of more general interaction potentials $W$ remain largely open.

These developments highlight the delicate interplay between nonlocal interaction forces and alignment mechanisms in shaping the asymptotic behavior of pressureless Euler dynamics.  
The main objective of this paper is to clarify this interplay in the presence of general interaction potentials.  

\smallskip

We focus on two representative classes of interaction potentials:
\begin{itemize}
\item In any spatial dimension, we consider interaction potentials that satisfy two-sided quadratic bounds.  
More precisely, we call $W:\R^d \to \R$ a \emph{$(\lambda,\Lambda)$-convex potential} if $W \in C^1(\R^d,\R)$ and there exist constants $0<\lambda \leq \Lambda < \infty$ such that
\begin{equation}\label{def:lam}
\lambda |x-y|^2 \le \langle \nabla W(x) - \nabla W(y), x-y \rangle \le \Lambda |x-y|^2 
\quad \text{for all } x,y \in \R^d.
\end{equation}
When $W \in C^2(\R^d,\R)$, this is equivalent to the uniform Hessian bounds
\[
\lambda I_d \preceq D^2 W(x) \preceq \Lambda I_d, \quad \text{for all } x \in \R^d.
\]
\item in one dimension, the repulsive--attractive Coulomb--quadratic potential
\bq\label{intro_W}
W(x) = -|x|+\frac12 x^2,
\eq
which lies beyond the uniformly convex framework.
\end{itemize}
For the communication weight $\phi$, we distinguish between bounded kernels
and weakly singular kernels of the form
\[
\phi(r)\sim r^{-\alpha}, \quad r\to0^+,
\]
with $\alpha>0$. 
This dichotomy turns out to be decisive for the rate of convergence: bounded kernels enforce exponential decay, whereas weakly singular ones produce algebraic rates. 

With these assumptions, the limiting flock profiles $(\rho_\infty,u_\infty)$ are characterized as follows.
The asymptotic velocity is given by conservation of momentum,
\begin{equation}\label{eq:uinfty}
u_\infty = \frac{1}{\|\rho_0\|_{L^1(\R^d)}}\int_{\R^d} \rho_0 u_0\,dx,
\end{equation}
while the density converges (up to translation) to a minimizer of the interaction energy
\[
\rho_\infty \in \argmin \lt\{ \iint_{\R^d\times\R^d} W(x-y)\rho(x)\rho(y)\,dxdy\,:\,\rho\in \mathcal P_2(\R^d)\rt\}.
\]
Here $\mathcal P_2(\R^d)$ denotes the set of probability measures with finite second moment. 
Since the interaction energy is translation invariant, the uniqueness of minimizers should be considered only up to translation; when needed, we use a centered representative. 

For instance, if $W$ is $\lambda$-convex with $\lambda>0$, then the unique minimizer up to translation is given by $\rho_\infty=\delta_0$.  
In contrast, for the Coulomb--quadratic potential \eqref{intro_W} on the real line, the unique minimizer up to translation is the uniform density on an interval, 
\[
\rho_\infty(x)=\frac12\,\mathbf{1}_{[-1,1]}(x),
\]
again up to translation. In higher dimensions, the analogue of this result is well known in potential theory, known as the Frost lemma: for Coulomb interactions combined with quadratic confinement, the minimizer is the uniform probability measure supported on a ball. This ``droplet profile'' plays the role of the equilibrium distribution in the theory of Coulomb and log gases; see, for instance, \cite{Ser19} for a detailed discussion. Thus, the one-dimensional interval law extends naturally to higher dimensions and a variety of perturbed Riesz potentials, see \cite{CMMRSV20,CS24} and the references therein, for instance.

\smallskip

 Our main contributions can be summarized in three representative cases:
\begin{itemize}
\item {\it One-dimensional dynamics with $(\lambda,\Lambda)$-convex potentials.}  
We remove all assumptions on the tail behavior of $\phi$ and establish quantitative convergence rates. 
For bounded kernels, exponential decay holds; for weakly singular kernels, we obtain sharp algebraic decay rates.  

\item {\it One-dimensional dynamics with the Coulomb--quadratic potential.}  
We construct perturbation variables around the asymptotic flock profile and derive stability estimates with explicit decay rates using hypocoercivity-type arguments.  
We show exponential convergence for bounded kernels and not necessarily sharp algebraic decay for weakly singular kernels.  

\item {\it Multi-dimensional dynamics with $(\lambda,\Lambda)$-convex potentials.}  
We combine convexity estimates with a hypocoercivity-type strategy to prove convergence. 
Exponential decay is established for bounded kernels and not necessarily sharp algebraic decay for weakly singular kernels. 
\end{itemize}

 A unifying feature of our results is that the {\it decay rate of convergence depends only on the local behavior of $\phi$ near the origin}: bounded kernels yield exponential decay, whereas weakly singular ones lead to algebraic rates. This provides a comprehensive description of the long-time asymptotics for Euler--alignment dynamics under broad classes of interaction potentials. To the best of our knowledge, these results constitute the first rigorous {\it quantitative characterization} of the large-time behavior of the Euler--alignment system with general interaction potentials and singular communication weights. In particular, the one-dimensional  $(\lambda,\Lambda)$-convex case with weakly singular kernels yields sharp algebraic decay rates, while for the Coulomb--quadratic potential and the multidimensional  $(\lambda,\Lambda)$-convex case, we establish only algebraic upper bounds.  The optimality of the latter remains an open problem, highlighting a key direction for future investigation. We emphasize that these results are of an {\it a priori} nature: they describe the asymptotic behavior of classical global solutions, without addressing the global regularity problem itself.

For convenience of notation, we extend the communication weight radially by
\[
\psi(x):=\phi(|x|), \quad x \in \R^d.
\]
Thus $\psi$ depends only on the radial variable, and throughout we do not distinguish between its radial profile on $[0,\infty)$ and its radial extension to $\R^d$.  
This convention will simplify the presentation of subsequent computations.  

Since the total mass is conserved in time, we may assume, without loss of generality, that $\rho$ is a probability density function, that is,
\[
\|\rho(t)\|_{L^1} = 1 \quad \text{for all } t \geq 0.
\]
In particular, under this normalization, the conserved asymptotic velocity is
\[
u_\infty=\int_{\R^d} \rho_0 u_0\,dx.
\]

Before stating our main results, we fix some notation.  
For probability measures $\mu,\nu$ on $\R^d$ with finite $p$-th moment, the $p$-Wasserstein distance is defined by
\[
 {\rm d}_p(\mu,\nu) := \inf_{\pi \in \Pi(\mu,\nu)} \lt(\iint_{\R^d\times\R^d} |x-y|^p \, d\pi(x,y)\rt)^{1/p},
\]
where $\Pi(\mu,\nu)$ denotes the set of all couplings of $\mu$ and $\nu$.  
In particular, we denote by $ {\rm d}_\infty(\mu,\nu)$ the $\infty$-Wasserstein distance
\[
 {\rm d}_\infty(\mu,\nu) := \inf_{\pi \in \Pi(\mu,\nu)} \, \sup_{(x,y)\in\supp \pi} |x-y|,
\]
which measures the maximal displacement between supports of $\mu$ and $\nu$.

For asymptotic notations, given two nonnegative functions $f(t)$ and $g(t)$, we write $f(t)\ls g(t)$ if there exists a constant $C>0$ independent of $t$ such that $f(t)\le C g(t)$ for $t \geq 0$. In this sense, $f(t)=\mathcal{O}(g(t))$ means $f(t)\lesssim g(t)$ for $t \geq 0$, while $f(t)=\Theta(g(t))$ as $t\to\infty$ indicates that both $f(t)\lesssim g(t)$ and $g(t)\lesssim f(t)$ hold.

%
%
%
%
%
\subsection{Main results}

We now formulate our main convergence theorems, which provide rigorous statements of the asymptotic behavior outlined in the preceding discussion. These results describe the long-time asymptotics of solutions under different structural assumptions on the interaction potential and the communication weight, and they reflect the delicate interplay between aggregation, dispersion, and alignment.  In particular, we identify conditions under which the density profile converges either to a traveling Dirac mass or to a nontrivial flock profile determined by the interaction potential, while the velocity field relaxes to a uniform asymptotic state. The analysis is divided into three representative cases, which together illustrate the range of possible emergent behaviors.  

Throughout, the center of mass trajectory evolves as
\[
\eta_c(t) :=\int_{\R^d} x \rho(t,x)\,dx.
\]
%
%
%
%
%
\subsubsection{$(\lambda,\Lambda)$-convex potentials in one dimension}

We begin with the one-dimensional case, where the interaction potential $W$ is assumed to be $(\lambda,\Lambda)$-convex in the sense of \eqref{def:lam}. In one dimension, this reduces to the two-sided slope condition
\bq\label{eq:lc:1d:ub}
 \lambda (x-y)  \le W'(x) - W'(y) \le \Lambda (x-y) \quad \text{for all }  x>y,
\eq
for some $0 < \lambda \leq \Lambda < \infty$. This restriction controls the strength of attraction and prevents overly fast collapse of the density profile, which in turn enables us to establish lower bounds on the decay rate.  

\begin{theorem}\label{thm_1:new}
Let $d=1$ and let $(\rho,u)$ be a global-in-time classical solution of the system \eqref{main_sys}. Suppose that the interaction potential $W$ satisfies \eqref{eq:lc:1d:ub} and $\rho_0$ is compactly supported. Define
\[
\rho_\infty(t,x) := \delta_{\eta_c(t)}(x).
\]
Then the following decay estimates hold:
\begin{itemize}
\item[(i)] If the communication weight function $\psi$ satisfies 
\bq\label{psi:type1}
0< \psi_{\rm m} \le  \psi(x) \le \psi_{\rm M}\quad \text{for all } |x| \le R
\eq
for some $R>0$, then we have
\[
 {\rm d}_\infty(\rho(t),\rho_\infty(t)) + \|u(t,\cdot)- u_\infty\|_{L^\infty({\rm supp}(\rho(t)))} \to 0
\]
exponentially fast as $t\to \infty$.
\item[(ii)] If the communication weight $\psi$ satisfies 
\bq\label{psi:type2}
\frac{A}{|x|^{\alpha}} \le \psi(x) \le \frac{B}{|x|^{\alpha}} \quad \text{for all } |x| \le R
\eq
for some $\alpha \in (0,1)$, $R >0$, then we have
\[
 {\rm d}_\infty(\rho(t),\rho_\infty(t)) = \Theta \lt(t^{-\frac{1}{\alpha}}\rt),\quad  \|u(t,\cdot)- u_\infty\|_{L^\infty({\rm supp}(\rho(t)))}= \mathcal{O}  \lt(t^{-\lt(\frac{1}{\alpha}-1\rt)}\rt) 
\]
as $t \to \infty$.
\end{itemize}
\end{theorem}

\begin{remark}
The assumptions on the communication weight $\psi$ in \eqref{psi:type1}--\eqref{psi:type2} are imposed only locally, thus allowing degeneracy; $\psi$ may vanish outside a bounded region. In our setting, this loss of Cucker--Smale interaction is compensated by the attractive potential $W$, which confines the agents. We note that \cite{DS21} also considered degenerate communication, where alignment on the one-dimensional torus was achieved by exploiting the periodic geometry rather than confinement.
\end{remark}

\begin{remark} If we only assume that $W$ is $\lambda$-convex, i.e., $W$ satisfies
\bq\label{eq:lc:1d}
 \lambda (x-y)  \le  W'(x) - W'(y)\quad \text{for all }  x>y,
\eq
in both cases $(i)$ and $(ii)$, we still have convergence although without quantitative rates:
\[
 {\rm d}_\infty(\rho(t),\delta_{\eta_c(t)}) + \|u(t,\cdot)- u_\infty\|_{L^\infty({\rm supp}(\rho(t)))} \to 0 \quad \mbox{as } t \to \infty.
\]
\end{remark}

 \begin{remark} 
If $W \in C^2(\R)$, then the assumptions \eqref{eq:lc:1d:ub} is equivalent to
\[
\lambda \leq W''(x) \leq \Lambda, \quad x \in \R.
\]
In this setting, exponential convergence was obtained in \cite{ST2020a} under the boundedness, together with a tail growth condition that
\[
\psi \in L^\infty(\R) \quad \mbox{and} \quad \limsup_{r \to \infty}  r \psi(r) = \infty.
\]
By contrast, Theorem \ref{thm_1:new} (i) shows that for bounded kernels, exponential decay follows without any such tail requirement. Moreover, while \cite{ST2020a} only treats the bounded case, our result further extends to weakly singular kernels, for which we obtain sharp algebraic decay rates as stated in part (ii).
\end{remark}

This theorem confirms the precise convergence rates mentioned in the introduction: exponential relaxation for bounded communication weights, and sharp algebraic decay for weakly singular kernels.

%
%
%
%
%
\subsubsection{Repulsive Coulomb potential with quadratic confinement in one dimension}

We next consider the one-dimensional case with the repulsive--attractive Coulomb--quadratic potential:
\bq\label{def_W_cc}
W(x) = -|x| + \frac12 x^2,
\eq
which lies outside the $\lambda$-convex framework.

In this case, the dynamics lead not to collapse into a Dirac mass, but rather to convergence towards a traveling density profile $\rho_\infty$ that preserves this equilibrium shape.  

\begin{theorem}\label{thm_2:new}
Let $d=1$ and let $(\rho,u)$ be a global-in-time classical solution of the system \eqref{main_sys}. Suppose that the interaction potential $W$ is given by \eqref{def_W_cc} and $\rho_0$ is compactly supported. Define
\[
\rho_\infty(t,x):= \frac{1}{2}\mathbf{1}_{\lt\{\eta_c(t)-1 \le x \le \eta_c(t)+1 \rt\}},
\] 
and we assume that 
\[
\psi(x) \ge \psi_{\rm m}>0\quad \text{for all } |x| \le R
\]
for some $R>0$. Then the following decay estimates hold:
\begin{itemize}
\item[(i)]  If  $\psi$ is bounded from above, then we have
\[
 {\rm d}_\infty(\rho(t),\rho_\infty(t))+  \|u(t,\cdot)- u_\infty\|_{L^\infty({\rm supp} (\rho(t)))} \to 0
\]
exponentially fast as $t\to \infty$.
\item[(ii)] If  $\psi$ satisfies  
\[
 \psi(x) \le \frac{B}{|x|^{\alpha}}\quad \mbox{for all $0\le |x| \le R$ for some $\alpha \in (0,1)$ and $R>0$}, 
\]
then we have
\[
 {\rm d}_2(\rho(t),\rho_\infty(t)) = \mathcal{O}\lt(t^{-(\frac{3}{4\alpha}-\frac{1}{2}) }\rt),\quad  \|u(t,\cdot)-u_\infty\|_{L^2(\rho(t))} =\mathcal{O} \lt(t^{- \lt\{(1-\alpha)(\frac{3}{4\alpha}-\frac{1}{2})\rt\} }\rt)
\]
as $t \to \infty$.
\end{itemize}
\end{theorem}

This theorem shows that the density $\rho(t)$ converges to a traveling flock profile of fixed width, translating with constant velocity $u_\infty$. The limiting state is the minimizer of the interaction energy associated with $W$, up to translation. For bounded communication weights, we obtain exponential convergence, while for weakly singular weights, we derive algebraic upper bounds. The sharpness of these algebraic rates, however, remains open.

%
%
%
%
%
\subsubsection{$(\lambda,\Lambda)$-convex potential in multi-dimensional setting}
In our final main result, we extend the one-dimensional analysis with $(\lambda,\Lambda)$-convex potentials to the multi-dimensional setting.  As in the one-dimensional case, the long-time dynamics lead to concentration of the density profile toward a traveling Dirac mass located at the center of mass.  
The essential new feature in higher dimensions is that convexity conditions on the interaction potential must now hold uniformly in all spatial directions. Intuitively, this requirement prevents the potential from becoming flat in some directions, which would otherwise correspond to forces that are too weak or anisotropic to stabilize the dynamics.  

We first recall that \eqref{def:lam} is equivalent to the two-sided inequality
\[
 \frac{\lambda}{2}|x-y|^2 \leq W(x) -W(y)- \lal \nabla W(y), x-y \ral  \le  \frac{\Lambda}{2}|x-y|^2 \quad \text{for all }  x,y \in \R^d
\]
for some $0<\lambda \le \Lambda <+\infty$. In particular, since $\nabla W(0)=0$, we obtain the quadratic lower and upper bounds
\bq\label{W_Lam}
\frac{\lambda}{2}|x|^2 \leq W(x) - W(0) \le \frac{\Lambda}{2}|x|^2 \quad \text{for all } x\in \R^d.
\eq
 These two-sided bounds place $W$ in the class of uniformly convex potentials with controlled growth: convex enough to ensure stability, yet not so steep as to overshadow the alignment effects.  

Within this framework, we establish quantitative convergence results in arbitrary dimensions.

\begin{theorem}\label{thm_3:new}Let $d\geq1$ and let $(\rho,u)$ be a global-in-time classical solution of the system \eqref{main_sys}. Suppose that the interaction potential $W(x)$ satisfies \eqref{def:lam}, and the communication weight $\psi$ satisfies
\bq\label{bdd_psi}
\psi(x) \geq \psi_{\rm m} > 0 \quad \mbox{for all } x \in \R^d.
\eq
Then the following decay estimates hold:\begin{itemize}
\item[(i)]  If $\psi$ is bounded from above, then we have
\[
 {\rm d}_2(\rho(t), \delta_{\eta_c(t)}) + \|u(t,\cdot)-u_\infty\|_{L^2(\rho(t))} \to 0
\]
exponentially fast as $t\to \infty$. 
\item[(ii)] If $\psi$ satisfies
\[
 \psi(x) \le \frac{B}{|x|^{\alpha}}\quad \mbox{for all $0\le |x| \le R$ for some $\alpha \in (0,2)$ and $R>0$}, 
\]
then we have
\[
 {\rm d}_2(\rho(t), \delta_{\eta_c(t)}) + \|u(t,\cdot)-u_\infty\|_{L^2(\rho(t))} =   \mathcal{O} \lt(t^{-\lt(\frac{1}{\alpha} - \frac{1}{2} \rt)} \rt)
\]
as $t \to \infty$.
\end{itemize}
\end{theorem}

\begin{remark}
Compared with Theorem \ref{thm_1:new} (ii), the assumption on $\psi$ in Theorem \ref{thm_3:new} (ii) allows for more singular kernels in the one-dimensional case.  
Moreover, while as a direct corollary of Theorem \ref{thm_1:new} (ii), we deduce
\[
 {\rm d}_2(\rho(t),\delta_{\eta_c(t)}) = \mathcal{O} \lt(t^{-\frac{1}{\alpha}} \rt),\quad  \|u(t,\cdot)- u_\infty\|_{L^2(\rho(t))}= \mathcal{O}\lt(t^{-\lt(\frac{1}{\alpha}-1\rt)}\rt) 
\]
the estimate of Theorem \ref{thm_3:new} provides an improved $L^2$ decay rate for the velocity field.  
This partially explains why, in the one-dimensional case, only an upper bound for the velocity decay could be obtained. We emphasize that these two results exploit completely different approaches; while Theorem \ref{thm_1:new} crucially takes advantage of the one-dimensional structure of the alignment system, the multi-dimensional result, Theorem \ref{thm_3:new}, makes use of energy estimates.
\end{remark}

\begin{remark}
In the special case of quadratic confinement $W(x)=\frac{\lambda}{2}|x|^2$ ($\lambda>0$), as mentioned before, the Euler--alignment dynamics reduce to an alignment system with a quadratic external potential. In that setting, for admissible communication weights with ``thin tails'' \eqref{thin}, it was shown in \cite[Lemma 2.1]{ST2020a} that the support diameter $D(t):={\rm diam}({\rm supp}\,\rho(t))$ remains uniformly bounded in time.  
As a consequence, we obtain a uniform lower bound on the communication weight along the flow,
\[
\psi(r)\ge \psi(D(t))\ge \psi_{\min}>0 \quad \text{for all } t\ge0,
\]
which verifies condition \eqref{bdd_psi}.  
Under this positivity, exponential flocking follows: $L^2$-flocking with an explicit rate, together with an $L^\infty$ counterpart, was established in \cite[Theorems 2.3 and 2.5]{ST2020a}.  

We emphasize that this mechanism relies essentially on the quadratic structure (via the harmonic-oscillator invariance) and is not available for general uniformly convex potentials. For $(\lambda,\Lambda)$-convex potentials beyond the quadratic case, \cite[Theorem~3.6]{ST2020a} obtained only algebraic $L^2$ decay, under the additional assumption that $\psi(0)$ is sufficiently large in terms of $\lambda$ and $\Lambda$ (with $\psi$ bounded). In contrast, Theorem~\ref{thm_3:new}(i) shows that exponential $L^2$ decay persists for all $(\lambda,\Lambda)$-convex potentials under the sole assumption \eqref{bdd_psi}, without any further tail or size condition on $\psi$.  The extension of the $L^\infty$ convergence beyond the quadratic case, however, remains open.

We also note that in \cite{ST2021b}, an anticipated Euler--alignment model was studied, where interactions are taken at anticipated positions $x^\tau = x + \tau u$. The system was recast into a generalized formulation with a matrix-valued communication kernel, and sub-exponential $L^2$-energy decay was proved under assumptions allowing decay of either $\psi$ or $D^2 W$ at infinity.
\end{remark}

%
%
%
%
%
%
%
%
%
\subsection{Organization of the paper}
The remainder of the paper is organized as follows. In Section \ref{sec:pre}, we collect preliminary properties of our main system. We begin with the conservation laws and the associated energy dissipation, and then reformulate the system in Lagrangian coordinates with suitable auxiliary variables. We also establish a Lagrangian--Eulerian correspondence via Wasserstein estimates, which serves as the link between characteristic estimates and Eulerian stability, and finally derive several Lipschitz-type estimates for the primitive of the communication weight, which will be used repeatedly in later sections. Section \ref{sec:ub} establishes uniform-in-time bounds for the diameter of the density support. Via the Lagrangian reformulation (Section \ref{sec:pre}), this reduces to controlling the maximal separation between characteristics. In one dimension, $\lambda$-convex interaction potentials enforce contraction of both the spatial (flow) and velocity diameters, whereas for the Coulomb--quadratic potential, we still obtain a uniform bound on the spatial diameter despite the loss of convexity. Notably, no upper bound on $W'$ is required to get the decay in the $\lambda$-convex case. Section \ref{sec:m1} is devoted to the one-dimensional dynamics under a $(\lambda,\Lambda)$-convex interaction potential. We derive a reduced system of ordinary differential inequalities for pairwise deviations of the characteristic variables and prove quantitative decay rates of the flow and velocity diameters: exponential in the case of bounded communication weights and algebraic in the weakly singular case. In Section \ref{sec:m2}, we address the one-dimensional Coulomb potential with quadratic confinement, which lies beyond the $\lambda$-convex framework. By introducing perturbation variables around the asymptotic profile, we establish dissipation estimates and show convergence to flocking states, again distinguishing between exponential and algebraic decay according to the type of communication weight. Section \ref{sec:m3} extends the analysis to the multidimensional setting with $\lambda$-convex interaction potentials. Using a hypocoercivity-type strategy, we prove exponential decay of $L^2$ deviations when the communication weight is bounded, and algebraic decay in the presence of singular weights. 
%
%
%
%
%
%
%
%
%

\section{Preliminaries}\label{sec:pre}

%
%
%
%
%
%
%
%
%
\subsection{Conservation laws and energy dissipation}
In this part,  we recall some fundamental physical properties of system \eqref{main_sys}. These basic identities highlight the conservative and dissipative mechanisms that govern the dynamics and will serve as the starting point for the large-time analysis.

First, the mass and momentum are conserved along the flow:
\[
\frac{d}{dt}\int_{\R^d} \rho(t,x)\,dx= 0,\quad \frac{d}{dt}\int_{\R^d} \rho(t,x)u(t,x)\,dx = 0.
\]
The center of mass of the system is defined by
\[
\eta_c(t) :=\int_{\R^d} x \rho(t,x)\,dx.
\]
Differentiating twice in time yields
\[
\frac{d^2}{dt^2}\eta_c(t) = \frac{d}{dt}\int_{\R^d} x\rho(t,x)\,dx = 0,
\]
and thus $\eta_c(t)$ evolves linearly. More explicitly, 
\bq\label{eq:eta:att}
\eta_c(t) = \int_{\R^d} x \rho_0(x)\,dx + t u_\infty,\quad u_\infty:=\int_{\R^d}u_0(x) \rho_0(x) \,dx.
\eq
This identity shows that the entire configuration drifts rigidly with the mean velocity $u_\infty$, while the interplay of interaction and alignment determines the relative positions of agents.

In addition to these conservation laws, the system possesses a natural energy-dissipation structure. Consider the combination of the kinetic fluctuations and the interaction energy:
\begin{align}\label{eq:diss}
\begin{aligned}
&\frac{d}{dt} \lt( \frac12\iint_{\R^d \times \R^d} |u(t,x) - u(t,y)|^2 \rho(t,x) \rho(t,y)\,dxdy + \iint_{\R^d \times \R^d} W(x-y) \rho(t,x) \rho(t,y)\,dxdy\rt) \cr
&\qquad = - \iint_{\R^d \times \R^d} \psi(x-y) |u(t,x) - u(t,y)|^2 \rho(t,x) \rho(t,y)\,dxdy \le 0.
\end{aligned}
\end{align}

Thus, the total energy decreases monotonically in time, with the alignment kernel $\psi$ being solely responsible for the dissipation. In particular, the quadratic velocity differences are damped at a rate determined by $\psi$, while the potential $W$ accounts for the conservative part of the dynamics. In the case $W\equiv 0$, once a uniform bound on the support diameter is available, the monotonicity of $\psi$ yields a positive lower bound, and the dissipation inequality \eqref{eq:diss} directly leads to exponential velocity alignment. When $W\neq 0$, the interaction energy must be controlled separately: under convexity assumptions, $W$ provides spatial coercivity, while the alignment term dissipates velocity fluctuations. In this situation, convergence estimates are typically obtained through hypocoercivity-type arguments \cite{Vil09}, which combine conservative and dissipative structures into a unified framework. This general structure, balancing conservation laws with dissipative alignment, will serve as the foundation of our large-time analysis.

%
%
%
%
%
%
%
%
%
\subsection{Lagrangian reformulation and auxiliary variables}\label{sec:lagr}
For the analysis of our system, it is advantageous to switch from the Eulerian to the Lagrangian description. Since there is no pressure term in \eqref{main_sys}, the dynamics along characteristics can be represented as a closed system for the flow map and velocity, where the evolution reduces to integro-differential relations in time. This structure is considerably simpler than the original Eulerian PDE, as the transport part is absorbed into the characteristic flow, and the nonlocal interactions appear explicitly in terms of the initial distribution $\rho_0$. Moreover, in the one-dimensional setting, the alignment operator admits a further simplification \cite{CCTT16,CCZ16,HKPZ19, ZZ20, CZ21}: it can be expressed through the primitive $\Psi$ of the weight $\psi$, which enables the introduction of auxiliary variables and leads to more explicit formulas for the asymptotic dynamics.

For the proofs of our main results, we make use of the Lagrangian flows associated with \eqref{main_sys}. We briefly explain the Lagrangian formulation below. 

\medskip
\noindent {\it General Lagrangian formulation.} We begin by defining the forward characteristics $\eta(t,x)$ for $x\in {\rm supp}(\rho_0)$ through
\[
\pa_t \eta(t,x) = u(t,\eta(t,x)) =: v(t,x), \quad \eta(x,0) = x.
\]
In terms of $(\eta,v)$, the system \eqref{main_sys} can be rewritten as
\begin{align}\label{main_sys2}
\begin{aligned}
&\pa_t \eta(t,x) = v(t,x),  \cr
&\pa_t v(t,x) = - (\nabla W \star \rho) (t,\eta(t,x)) - \int_{\R^d} \psi(\eta(t,x)-y)(v(t,x) - u(t,y))\rho(t,y)\,dy.
\end{aligned}
\end{align}
Since the flow map $\eta(t,\cdot)$ transports the initial distribution $\rho_0$, the density $\rho(t)$ at time $t$ is given as the pushforward of $\rho_0$ under $\eta$, namely $\rho(t) = \eta(t,\cdot) \# \rho_0$. That is, for every test function $\varphi\in C_c(\R^d)$ we have
\[
\int_{\R^d} \varphi(x) \rho(t,x)\,dx = \int_{\R^d} \varphi(\eta(t,x))\rho_0(x)\,dx.
\]
This identity represents the mass conservation of the transport equation and provides a direct way to rewrite all convolution terms in terms of the initial density $\rho_0$. In particular,
\[
(\nabla W \star \rho) (t,\eta(t,x)) = \int_{\R^d} \nabla W(\eta(t,x) - \eta(t,y))\rho_0(y)\,dy
\]
and
\[
\int_{\R^d} \psi(\eta(t,x)-y)(v(t,x) - u(t,y))\rho(t,y)\,dy = \int_{\R^d} \psi(\eta(t,x)-\eta(t,y))(v(t,x) - v(t,y))\rho_0(y)\,dy.
\]

\medskip
\noindent {\it Auxiliary variable and simplified structure in 1D.} In the one-dimensional setting, the structure of the alignment term simplifies considerably. Indeed, we verify that
\[
\int_\R \psi(\eta(t,x)-\eta(t,y))(v(t,x) - v(t,y))\rho_0(y)\,dy = \frac{\pa}{\pa t}\int_\R \Psi(\eta(t,x)-\eta(t,y))\rho_0(y)\,dy,
\]
where $\Psi$ denotes a primitive of $\psi$, that is,
\[
\Psi(r) = \int_0^r \psi(s)\,ds.
\]
This observation suggests that the cumulative effect of alignment can be absorbed into a new variable. We therefore introduce the auxiliary Lagrangian quantity
\bq\label{def:om}
\omega(t,x) := v(t,x) + \int_\R \Psi(\eta(t,x)-\eta(t,y))\rho_0(y)\,dy,
\eq
which can be regarded as the modified velocity along the characteristic. With this definition, the system
\begin{align}\label{main_sys3}
\begin{aligned}
&\pa_t \eta(t,x) = \omega(t,x) - \int_\R \Psi(\eta(t,x)-\eta(t,y))\rho_0(y)\,dy,\cr
&\pa_t \omega(t,x) =  - \int_\R \pa_x W(\eta(t,x) - \eta(t,y))\rho_0(y)\,dy.
\end{aligned}
\end{align}
This reformulation highlights the advantage of working in Lagrangian coordinates: the velocity equation reduces to a purely nonlocal forcing term, while the alignment contribution has been absorbed into the auxiliary variable $\omega$.

\medskip
\noindent{\it Asymptotic density.}
Depending on the interaction potential $W$, the asymptotic density profile $\rho_\infty$ and the corresponding asymptotic Lagrangian map $\eta_\infty$ are defined differently. We remind the reader that the asymptotic density profile defined in the introduction is given by
\begin{align}\label{asymprho}
\rho_\infty(t,x) = \begin{cases}
    \delta_{\eta_c(t)}(x) \quad &\text{if }W(x)\text{ is $\lambda$-convex,}\\[3mm]
    \frac{1}{2}1_{\lt\{\eta_c(t)-1 \le x \le \eta_c(t)+1\rt\}}\quad &\text{if }W(x)=-|x| + \frac{x^2}2 \text{ when $d=1$.}
\end{cases}
\end{align}

\medskip
\noindent{\it Case of $\lambda$-convex potentials.}
If $W$ is a $\lambda$-convex potential with $\lambda>0$ (Sections \ref{sec:m1} and \ref{sec:m3}), we set
\[
\eta_\infty(t,x) := \eta_c(t),
\]
where $\eta_c(t)$ denotes the center of mass defined in \eqref{eq:eta:att}. 
In this case, the limiting density is simply a Dirac mass located at $\eta_c(t)$:
\[
\rho_\infty(t) = \eta_\infty(t,\cdot)\#\rho_0 = \delta_{\eta_c(t)}.
\]
Thus both the time-dependent density $\rho(t)$ and its limiting profile $\rho_\infty(t)$ can be written as push-forwards of $\rho_0$:
\[
\rho(t) = \eta(t,\cdot)\#\rho_0, \quad \rho_\infty(t) = \eta_\infty(t,\cdot)\#\rho_0.
\]
Introducing 
\[
\tilde\eta(t,x):=\eta(t,x)-\eta_\infty(t,x), 
\]
we note that
\[
\int_{\R^d} \tilde\eta(t,x)\,\rho_0(x)\,dx = 0.
\]
In the one-dimensional case and with these definitions, the evolution equations take the form
\[
\begin{aligned}
\pa_t \tilde \eta(x)  & = \omega(x)  - \int_\R \Psi(\eta(x)-\eta(y)) \rho_0(y)\,dy - u_\infty, \cr
\pa_t \omega(x)   & =  - \int_\R \pa_x W(\eta(t,x) - \eta(t,y))\rho_0(y)\,dy.
\end{aligned}
\]

\medskip
\noindent {\it Case of 1D Coulomb potential with quadratic confinement.} We now specialize to the case
\[
W(x) = -|x| + \frac12 x^2,
\]
which combines the repulsive Coulomb potential with a quadratic confinement. In this case, we set
\[
\eta_\infty(t,x) := \eta_c(t) + \eta_r(x),
\]
where $\eta_c(t)$ denotes the center of mass defined in \eqref{eq:eta:att} and $\eta_r(x)$ satisfies
\[
\frac{1}{2}1_{\{-1\le x \le 1\}} = \eta_r(t,x) \# \rho_0.
\]
Here $\eta_c(t)$ describes the drift of the center of mass, while $\eta_r(x)$ represents the stationary reference profile determined solely by the initial distribution $\rho_0$.
Indeed, we have the explicit formula
\[
\eta_r(x)= 2\int_{-\infty}^x \rho_0(y)\,dy -1.
\]
Combining this identity with the center-of-mass trajectory \eqref{eq:eta:att}, we obtain the explicit representation
\begin{align*}
\begin{aligned}
\eta_\infty(t,x) &= \lt( \int_\R y\rho_0(y)\,dy +  t\int_\R u_0(y)\rho_0(y)\,dy\rt) + \lt(2\int_{-\infty}^x \rho_0(y)\,dy - 1 \rt) .
\end{aligned}
\end{align*}

As in the $\lambda$-convex potential case, both the time-dependent density $\rho(t)$ and its limiting profile $\rho_\infty(t)$ can be written as push-forwards of $\rho_0$:
\[
\rho(t) = \eta(t,\cdot)\#\rho_0, \quad \rho_\infty(t) = \eta_\infty(t,\cdot)\#\rho_0.
\]
Denoting by $
\tilde\eta(t,x):=\eta(t,x)-\eta_\infty(t,x), 
$
we find that
\[
\int_{\R^d} \tilde\eta(t,x)\,\rho_0(x)\,dx = 0.
\]
In this setting we compute
$$\begin{aligned}
\int_\R \pa_x W(\eta(x) - \eta(y))\rho_0(y)\,dy 
& = -\int_\R {\rm sgn}(\eta(x) - \eta(y))\rho_0(y)\,dy + \int_\R (\eta(x) - \eta(y))\rho_0(y)\,dy\cr
&\quad = \eta(x) - \int_\R {\rm sgn}(x-y)\rho_0(y)\,dy - \int_\R \eta(y)\rho_0(y)\,dy\cr
&\quad = \eta(x) - \eta_\infty(x).
\end{aligned}$$
Here we used the fact that
\bq\label{sgn}
{\rm sgn}(\eta(x) - \eta(y)) = {\rm sgn}(x-y), \quad x,y \in {\rm supp}(\rho_0),
\eq
which follows from the order-preserving property of characteristics.  Moreover, we find
\[
\int_\R {\rm sgn}(x-y)\rho_0(y)\,dy = \int_{-\infty}^x \rho_0(y)\,dy - \int_x^\infty\rho_0(y)\,dy = 2\int_{-\infty}^x \rho_0(y)\,dy - 1.
\]
 Differentiating yields
\bq\label{eq:uinfty:new}
\pa_t \eta_\infty(t,x) = \pa_t \eta_c(t) = \int_\R u_0(y)\rho_0(y)\,dy =:u_\infty
\eq
and
\[
\pa_x \eta_\infty(t,x) = \pa_x \eta_r(x) = 2\rho_0(x).
\]
In particular, $u_\infty$ coincides with the conserved mean velocity introduced earlier in \eqref{eq:uinfty}, confirming the consistency of this representation.

It is therefore convenient to describe the dynamics in terms of fluctuations around the asymptotic profile $\eta_\infty$. To this end, we set
\[
\tilde \eta := \eta - \eta_\infty
\]
and introduce the modified velocity
\bq\label{eq:tom}
\tilde \omega(x) := v(x) + \int_\R (\Psi(\eta(x)-\eta(y)) - \Psi(\eta_\infty(x)-\eta_\infty(y)))\rho_0(y)\,dy.
\eq
With these definitions, the evolution equations take the form
\[
\begin{aligned}
\pa_t \tilde \eta(x)  & = \tilde\omega(x)  - \int_\R (\Psi(\eta(x)-\eta(y)) - \Psi(\eta_\infty(x)-\eta_\infty(y)))\rho_0(y)\,dy - u_\infty, \cr
\pa_t \tilde \omega(x)   & = - \tilde \eta(x).
\end{aligned}
\]
This formulation shows that, in the Coulomb-quadratic case, the variable $\tilde\omega$ evolves according to a harmonic-type relation with $\tilde\eta$, while the nonlocal alignment effect is expressed through the difference of the $\Psi$-terms. Such a reformulation will be essential for analyzing the long-time asymptotics.

Summarizing the discussions above, the asymptotic Lagrangian map $\eta_\infty$ is given as
\begin{align}\label{asymp}
\eta_\infty(t,x) = \begin{cases}
    \eta_c(t) \quad &\text{if }W(x) \text{ is $\lambda$-convex,}\\[3mm]
    \eta_c(t) + 2\displaystyle\int_{-\infty}^x \rho_0(y)\,dy -1 \quad &\text{if }W(x)=-|x| + \frac{x^2}2 \text{when $d=1$}.
    \end{cases}
\end{align}

Finally, in order to quantify the maximal pairwise deviations of the characteristic variables, we introduce a unified notation for diameters. For any function $h: [0,\infty) \times {\rm supp}(\rho_0)\to\R^d$, we set
\[
D_h(t): = \sup_{x,y \in {\rm supp}(\rho_0)}|h(t,x) - h(t,y)|.
\]
In particular, $D_\eta(t)$, $D_\omega(t)$, and $D_v(t)$ denote the diameters associated with the flow map $\eta$, the auxiliary variable $\omega$, and the velocity $v$, respectively. 
By construction, $D_\eta(t)$ coincides with the diameter of the density support ${\rm supp}(\rho(t))$, since $\rho(t) = \eta(t,\cdot)\#\rho_0$ and $\eta(t,\cdot)$ is a measure-preserving bijection between 
${\rm supp}(\rho_0)$ and ${\rm supp}(\rho(t))$. Likewise, $D_v(t)$ represents the maximal oscillation of the Eulerian velocity $u(t,\cdot)$ over ${\rm supp}(\rho(t))$, and $D_\omega(t)$ admits a similar interpretation for the auxiliary field. More generally, for any such variable, we have
\[
D_h(t) = \sup_{x,y \in {\rm supp}(\rho(t))} | h(t,\eta^{-1}(t,x)) - h(t,\eta^{-1}(t,y))|.
\]
These quantities thus provide convenient measures of spatial spread and velocity fluctuations. 
Their decay will play a central role in the large-time analysis.

%
%
%
%
%
%
%

\subsection{Lagrangian--Eulerian correspondence via Wasserstein estimates} 
In the previous subsections, we introduced the conservative and dissipative structures of the system, together with its Lagrangian reformulation and auxiliary variables. 
It remains to clarify how estimates obtained for the Lagrangian flow map $\eta(t,\cdot)$ can be transferred into the Eulerian framework, where the density is given by $\rho(t)=\eta(t,\cdot)\#\rho_0$. For this purpose, we first characterize the asymptotic Eulerian density corresponding to each type of interaction potential $W$, and then recall a standard estimate that relates push-forward measures to Wasserstein distances. 
This Lagrangian--Eulerian correspondence will serve as the bridge between the dynamical estimates obtained for characteristic variables and the Eulerian stability statements in our main theorems.

\medskip
\noindent{\it Comparison via Wasserstein distance.}
The following lemma provides a general way to compare push-forwards of the same measure. It will allow us to transfer decay estimates for $\tilde\eta$ into quantitative convergence results for $\rho(t)$ in Wasserstein distances.

\begin{lemma}\label{lem:Wst}
Let $\rho_0 \in \mathcal{P}(\R^d)$ and let $T_1, T_2:\R^d\to\R^d$ be Borel measurable maps. Set
\[
\rho_1:=T_1\#\rho_0, \quad \rho_2:=T_2\#\rho_0, \quad \mbox{and} \quad \tilde T :=T_1 - T_2.
\]
Then the following bounds hold:
\begin{itemize}
\item[(i)] For $p \in [1,\infty)$, if $\tilde T \in L^p(\rho_0)$, then
\[
 {\rm d}_p(\rho_1,\rho_2) \leq \lt(\int_{\R^d} |\tilde T(x)|^p\, \rho_0(dx) \rt)^{\frac1p}
\]
and for $p = \infty$
\[
 {\rm d}_\infty(\rho_1,\rho_2) \leq \|\tilde T(x)\|_{L^\infty(\rho_0)}
\]
In particular, if
\[
\int_{\R^d} \tilde T(x) \,\rho_0(dx) =0,
\]
then 
\[
 {\rm d}_2^2(\rho_1, \rho_2) \leq \frac12 \iint_{\R^d \times \R^d} |\tilde T(x) - \tilde T(y)|^2 \,\rho_0(dx)\rho_0(dy).
\]
\item[(ii)] Assume that $T_2 \equiv x_c \in \R^d$ is constant with 
\[
x_c = \int_{\R^d} T_1(x)\, \rho_0(dx),
\]
and define the (essential) diameter
\[
D_{T_1} := \esssup_{x,y \sim \rho_0} |T_1(x) - T_1(y)|. 
\]
If $D_{T_1} < \infty$, then we have
\[
\frac12 D_{T_1} \leq  {\rm d}_\infty(\rho_1, \delta_{x_c}) = \esssup_{x}|T_1(x) - x_c| \leq D_{T_1}.
\]
In particular, $ {\rm d}_\infty(\rho_1, \delta_{x_c})$ is quantitatively equivalent to the diameter of $T_1({\rm supp}(\rho_0))$.
\end{itemize}
\end{lemma}
\begin{proof}
(i) Let $\pi=(T_1, T_2)\#\rho_0$. Then $\pi \in \Pi(\rho_1, \rho_2)$. Thus, for $1 \leq p < \infty$, we have
\[
 {\rm d}_p^p(\rho_1,\rho_2)  \leq \int_{\R^d} |T_1(x)-T_2(x)|^p\,\rho_0(dx),
\]
which yields the assertion after taking $p$-th roots. The case $p=\infty$ follows by taking the limit $p\to\infty$.

We also find for $p=2$
\begin{align*}
\int_{\R^d} |\tilde T(x)|^2\, \rho_0(dx) &= \frac12 \iint_{\R^d \times \R^d} \lt(|\tilde T(x)|^2 + |\tilde T(y)|^2 \rt) \rho_0(dx) \rho_0(dy) \cr
&=\frac12 \iint_{\R^d \times \R^d} |\tilde T(x) - \tilde T(y)|^2 \, \rho_0(dx) \rho_0(dy) + \lt|\int_{\R^d} \tilde T(x)\,\rho_0(dx) \rt|^2.
\end{align*}
Thus, if $\int_{\R^d} \tilde T(x)\,\rho_0(dx) = 0$, we get the symmetric form. 

\medskip
\noindent (ii) We first notice that
\[
x_c = \int_{\R^d} T_1(x)\, \rho_0(dx) = \int_{\R^d} x \rho(dx)
\]
and ${\rm supp}(\rho) \subset T_1({\rm supp}(\rho_0))=:A$. 

For the upper bound, since $x_c \in {\rm conv} (A)$, for every $z \in A$, we have $|z-x_c| \leq \sup_{z,z' \in A}|z-z'|$. This implies
\[
 {\rm d}_\infty(\rho_1, \delta_{x_c})   \leq D_{T_1}.
\]
Next, by the triangle inequality, for all $x,y \in \R^d$, we get
\[
|T_1(x) - T_1(y)| \leq |T_1(x) - x_c| + |T_1(y) - x_c|.
\]
Taking essential suprema on both sides yields
\[
D_T \leq 2 \esssup_{x \sim \rho_0} |T_1(x) - x_c|.
\]
Since $x_c \in {\rm conv}(A)$, this further gives
\[
D_T \leq 2  {\rm d}_\infty(\rho_1, \delta_{x_c}).
\]
This completes the proof.
\end{proof}

%
%
%
%
%
%
%
%
%
\subsection{Lipschitz-type estimates for the primitive $\Psi$}

In this subsection, we collect several Lipschitz-type estimates for the primitive function $\Psi$, which will play a key role in controlling alignment terms in the Lagrangian one-dimensional formulation. Recall that  
\[
\Psi(r) = \int_0^r \psi(s)\,ds,
\]
where $\psi$ is nonnegative, radially decreasing, and locally integrable. Throughout this subsection, we omit the explicit time dependence of $\eta(t,\cdot)$ for notational convenience.

The following results quantify the increments of $\Psi$ both from below and from above, and provide scale-adapted Lipschitz bounds depending on the diameter of the one-dimensional flow map $\eta(t,x)$ and the structure of $\psi$.

\begin{lemma}\label{lem:Psi}
For each $x,y,z \in {\rm supp}(\rho_0)$ with $x>y$, we have
\bq\label{eq:Psi:lb}
\Psi(\eta(t,x)-\eta(t,z)) - \Psi(\eta(t,y)-\eta(t,z)) \ge \psi(D_\eta(t))( \eta(t,x) - \eta(t,y))
\eq
and
\bq\label{eq:Psi:ub}
\Psi(\eta(t,x) -\eta(t,z)) - \Psi(\eta(t,y)- \eta(t,z)) \le 2 \Psi\lt(\frac{1}{2}(\eta(t,x)-\eta(t,y)) \rt).
\eq
\end{lemma}
\begin{proof}
For the lower bound, by applying the mean value theorem, we get
\[
\Psi(\eta(x)-\eta(z)) - \Psi(\eta(y)-\eta(z)) = \psi(\xi(x,y,z))\lt(\eta(x) - \eta(y)\rt),
\]
where $\xi(x,y,z)$ is given as
\[
\xi(x,y,z) = \tau \lt(\eta(x)-\eta(z) \rt) + (1 - \tau) (\eta(y)-\eta(z))
\]
for some $\tau \in [0,1]$. Since $|\xi(x,y,z)| \le D_\eta(t)$ and $\psi$ is nonincreasing,
the inequality \eqref{eq:Psi:lb} follows.

For the upper bound, we distinguish three cases: $\eta(z) < \eta(y)$, $\eta(y) \le  \eta(z) \le \eta(x)$, and $\eta(z) > \eta(x)$. If $\eta(z)<\eta(y)$ or $\eta(z)>\eta(x)$, then we obtain
\[
\Psi(\eta(x) - \eta(z)) - \Psi(\eta(y)-\eta(z)) = \int_{\eta(y)-\eta(z)}^{\eta(x) - \eta(z)} \psi(r)\,dr \le \int_0^{\eta(x) - \eta(y)} \psi(r)\,dr = \Psi(\eta(x) - \eta(y)).
\]
On the other hand, if $\eta(y)\le \eta(z)\le \eta(x)$, then
\begin{align*}
\Psi(\eta(x) - \eta(z)) - \Psi(\eta(y)-\eta(z)) &= \int_0^{\eta(x)-\eta(z)} \psi(r)\,dr +  \int_0^{\eta(z)-\eta(y)} \psi(r)\,dr \\
&= \Psi(\eta(x) - \eta(z)) + \Psi(\eta(z) - \eta(y)).
\end{align*}
Then it follows that
\[
\Psi(\eta(x) - \eta(y)) \le \Psi(\eta(x) - \eta(z)) + \Psi(\eta(z) - \eta(y)) \le 2\Psi \lt( \frac{1}{2} (\eta(x) - \eta(y)) \rt)
\]
since $[0,\infty) \ni r \mapsto \Psi(r)$ is concave. This completes the proof.
\end{proof}

The previous lemma provides a general Lipschitz-type control, with coefficients depending on the global diameter $D_\eta$. Under additional assumptions on the communication weight $\psi$, these bounds can be refined. In the bounded-weight case, the estimates reduce to expressions depending solely on the local increment $\eta(x)-\eta(y)$. In the singular case, we still obtain a lower bound involving $D_\eta$, while the upper bound can be expressed purely in terms of the increment.
\begin{corollary}\label{cor:Psi}
Let $\alpha \in (0,1)$. Suppose $x,y,z \in {\rm supp}(\rho_0)$ with $x>y$ and $|\eta(x)-\eta(y)|\le R$. Then, the following Lipschitz-type estimates hold:
\begin{itemize}
\item [(i)] If $\psi$ satisfies \eqref{psi:type1}, then 
\[
\psi_{\rm m}(\eta(x)-\eta(y)) \le \Psi(\eta(x)-\eta(z)) - \Psi(\eta(y)-\eta(z)) \le \psi_{\rm M}(\eta(x)-\eta(y)).
\]
\item [(ii)] If $\psi$ satisfies \eqref{psi:type2}, then 
\[
A(D_\eta)^{-\alpha}(\eta(x)-\eta(y)) \le \Psi(\eta(x)-\eta(z)) - \Psi(\eta(y)-\eta(z)) \le \frac{2^\alpha B}{1-\alpha}(\eta(x)-\eta(y))^{1-\alpha}.
\]
\end{itemize}
\end{corollary}
\begin{proof}
The lower bounds are immediate from \eqref{eq:Psi:lb} under assumptions \eqref{psi:type1}--\eqref{psi:type2}. For the upper bounds, observe that for  weights of type \eqref{psi:type1},
\[
\Psi(r) \le \psi_{\rm M} r,
\]
while for type \eqref{psi:type2},
\[
\Psi(r) \le \frac{B}{1-\alpha} r^{1-\alpha}, \quad 0 \leq r \leq R.
\]
Substituting these into \eqref{eq:Psi:ub} yields the result.
\end{proof}

Finally, we provide a complementary averaged estimate, which compares increments at two different points and shows that the effective Lipschitz constant of $\Psi$ decreases with scale.
\begin{lemma}\label{lem:Psi:avg}
Let $\Psi:\R \to \R$ be a $C^1$ odd function, increasing and concave on $[0,\infty)$. Then for any $a,b \in \R$ with $ab>0$, the increment of $\Psi$ satisfies the averaged Lipschitz estimate
\[
|\Psi(a) -\Psi(b)| \le \min\lt\{\frac{\Psi(a)}{a}, \frac{\Psi(b)}{b}\rt\}|a-b|.
\]
Moreover, the map $r \mapsto \frac{\Psi(r)}{r}$ is decreasing on $(0,\infty)$.
\end{lemma}
\begin{proof}
Without loss of generality, assume $a>b>0$. By the mean value theorem and the concavity of $\Psi$, we obtain
\[
0 \le \frac{\Psi(a)-\Psi(b)}{a-b} \le \frac{\Psi(b) - \Psi(0)}{b-0}=\frac{\Psi(b)}{b}.
\]
In particular, we obtain
\[
0 \le \frac{\Psi(a)-\Psi(b)}{a-b} \le \frac{\Psi(a)}{a} \le \frac{\Psi(b)}{b},
\]
which proves the desired inequality and shows that $r\mapsto \Psi(r)/r$ is decreasing.
\end{proof}

%
%
%
%
%
%
%
%
%
\section{Uniform-in-time bounds for diameter of density support}\label{sec:ub}
In this section, we establish uniform-in-time bounds for the diameter of the density support. 
As explained in Section \ref{sec:lagr}, the quantity $D_\eta(t)$ coincides with the diameter of 
${\rm supp}(\rho(t))$, since $\rho(t)$ is obtained as the push-forward of $\rho_0$ under the 
Lagrangian flow $\eta(t,\cdot)$. Thus, controlling the support diameter is equivalent to 
controlling the maximal distance between characteristics. In particular, the study of $D_\eta(t)$ reduces to analyzing the evolution of extreme trajectories and their velocity differences. 

The behavior of $D_\eta(t)$ crucially depends on the convexity properties of the interaction potential. For $\lambda$-convex potentials, convexity induces contraction and guarantees $D_\eta(t)\to0$ as $t\to\infty$. In contrast, for the Coulomb potential with quadratic confinement, convexity is lost and genuine contraction fails, but we will still prove a uniform-in-time bound. The main proposition below summarizes both cases.

\begin{proposition}\label{prop:1dlc:0}
Let $d=1$ and let $(\eta,v)$ be a global solution to \eqref{main_sys2} with sufficient regularity.  Suppose that the communication weight function $\psi$ is a nonnegative, radially decreasing, locally integrable function and $\rho_0$ is compactly supported.
\begin{itemize}
\item[(i)] If the interaction potential $W$ satisfies the one-dimensional $\lambda$-convexity condition \eqref{eq:lc:1d}, then flow and velocity diameters decay to zero:
\[
D_\eta(t) + D_\omega(t) + D_v(t) \to 0 \quad \mbox{as } t \to \infty.
\]
\item[(ii)] If the interaction potential $W$ is given by \eqref{def_W_cc}, then flow diameter is uniformly bounded in time:
\[
D_\eta(t) \le 2 + \max\lt\{ \sqrt{(D_\eta(0)-2)^2+D_\omega(0)^2}, \ \max\{D_{\omega}(0),2\Psi(1)+1\}+2\rt\}=:\overline{D}
\]
for all $t\geq0$.
\end{itemize}
\end{proposition}

\begin{remark}The upper bound assumption on the quotient $W'$ in \eqref{eq:lc:1d:ub} is not required to get the decay estimate in Proposition \ref{prop:1dlc:0} (i).
\end{remark}

%
%
%
%
%
%
%
%
%

\subsection{$\lambda$-convex potential case}

We begin with the case of $\lambda$-convex interaction potentials, corresponding to part (i) of Proposition \ref{prop:1dlc:0}. In this setting, convexity provides a coercive mechanism ensuring contraction of trajectories. The argument reduces to analyzing the relative separation and velocity difference of two characteristics, which satisfy a closed system of differential inequalities. 

In the sequel, we fix $x,y\in {\rm supp}(\rho_0)$ with $x>y$. Then, we have the monotonicity of characteristics:
\bq\label{eq:safe}
\eta(t,x) >\eta(t,y),\quad \text{for all } t\ge0,
\eq
since characteristics never collide.

Combining the reformulated system \eqref{main_sys3} with the $\lambda$-convexity of $W$ \eqref{eq:lc:1d}, we obtain
\begin{align}\label{eq_main_2}
\begin{aligned}
\pa_t (\eta(t,x) - \eta(t,y) )&= \omega(t,x) - \omega(t,y) \cr
&\quad - \lt(\int_\R \lt(\Psi(\eta(t,x)-\eta(t,z)) - \Psi(\eta(t,y)-\eta(t,z)) \rt)\rho_0(z)\,dz\rt),\cr
\pa_t ( \omega(t,x) -  \omega(t,y)) &\le  - \lambda ( \eta(t,x) - \eta(t,y)).
\end{aligned}
\end{align}
Since $\Psi$ is an increasing function, the first equation in \eqref{eq_main_2} yields a simple upper bound
\bq\label{eq:facile}
\pa_t (\eta(t,x) - \eta(t,y) ) \le \omega(t,x) - \omega(t,y). 
\eq

These inequalities suggest studying an auxiliary two-dimensional system of ordinary differential inequalities (ODIs). Although the system is elementary, it exhibits strong contraction properties which are sufficient to conclude the proof of Proposition \ref{prop:1dlc:0}.

\begin{lemma}\label{lem:1dlc:basic}
Let $\kappa \geq 0$ and let $(X,Y)$ be a $C^1([0,\infty);(0,\infty) \times \R)$ curve satisfying  
\begin{equation}\label{ODI_0}
\begin{split}
\kappa (Y(t) - f(X(t)) ) \leq  \dot{X}(t) &\le Y(t), \\
\dot{Y}(t) &\le -\lambda X(t),
\end{split}
\end{equation}
for some $\lambda>0$ and an increasing continuous function $f: [0,\infty) \to [0,\infty)$ with $f(0) = 0$. Then the following properties hold:
\begin{itemize}
\item[(i)] $Y(t)$ is strictly decreasing and $Y(t)>0$ for all $t\ge 0$.
\item[(ii)] $\lim_{t\to \infty} X(t) =0$.
\item[(iii)] If $\kappa > 0$, then $\lim_{t \to \infty} Y(t) =0$.
\end{itemize}
\end{lemma}

Before proving the lemma, let us emphasize the role of the lower bound in \eqref{ODI_0}. 
If $\kappa=0$, then the condition $\dot X \ge 0$ merely ensures that $X$ is nondecreasing. 
In this case, \eqref{ODI_0}$_2$ still implies that $Y(t)$ is strictly decreasing and converges to some nonnegative limit $Y_*\ge 0$, but it does not rule out the possibility that $Y_*>0$. 
By contrast, when $\kappa>0$ the inequality $\dot X \ge \kappa(Y-f(X))$ prevents such a scenario: if $Y$ were to converge to a positive limit, the inequality would force $\dot X$ to remain strictly positive for large $t$, which is incompatible with the convergence of $X(t)$ to zero. 
Hence, the additional lower bound on $\dot X$ is precisely what guarantees that $Y(t)$ vanishes in the large-time limit. 
In the context of \eqref{eq_main_2}, this condition reflects the contribution of the alignment term, which is expressed through the primitive function $\Psi$ and plays an essential role in the asymptotic contraction of the system.

\begin{proof}[Proof of Lemma \ref{lem:1dlc:basic}]
(i) The monotonicity of $Y$ follows directly from \eqref{ODI_0}$_2$ together with the positivity of $X$. For the positivity of $Y$, suppose that $Y(t_0) \le 0$ for some $t_0 \ge 0$. Since $Y$ is strictly decreasing, there exists $t_1>t_0$ such that $Y(t_1) <0$. From $\eqref{ODI_0}_1$, we then obtain 
\[
\dot{X}(t) \le Y(t) \le Y(t_1) <0, \quad \text{for all }  t \ge t_1,
\]
which leads to $X(t) = 0$ in a finite time for some $t > t_1$, leading to a contradiction. 

\medskip
\noindent (ii) From \eqref{ODI_0}$_2$, we deduce
\[
\int_0^t X(s)\,ds \le \frac{1}{\lambda}\lt(Y(0)-Y(t)\rt) \le \frac{Y(0)}{\lambda} <+\infty,
\]
proving $X \in L^1([0,\infty))$. On the other hand, \eqref{ODI_0}$_1$ yields
\[
\dot{X}(t) \le Y(t) \le Y(0) \quad \mbox{for all } t \geq 0,
\]
hence $X$ is uniformly continuous on $[0,\infty)$. It is a standard fact that a uniformly continuous, nonnegative, integrable function must converge to zero as $t\to\infty$. Therefore $X(t)\to 0$ as claimed.

\medskip
\noindent (iii) By (i), $Y(t)$ is positive and strictly decreasing, hence the limit $Y_* = \lim_{t \to \infty} Y(t)$ exists with $Y_*\ge 0$. Assume $Y_*>0$. Since $X(t)\to 0$ as $t\to\infty$ and $f$ is continuous with $f(0)=0$, there exists $t_0 \geq 0$ such that $f(X(t)) < \frac{Y_*}{2}$ for all $t\ge t_0$. Then from \eqref{ODI_0}$_1$ we obtain
\[
\dot{X}(t) \ge \kappa( Y(t) - f(X(t))) \ge  \frac{\kappa Y_*}{2}>0\quad \text{for all }  t \ge t_0,
\]
which contradicts $X(t)\to 0$. Hence $Y_*=0$.
\end{proof}

The proof relies on reducing the dynamics of \eqref{eq_main_2} to a two-dimensional system of differential inequalities involving the distance of two characteristics and the corresponding velocity difference. This reduction makes it possible to apply Lemma \ref{lem:1dlc:basic}, which immediately yields the decay of $D_\eta$ and $D_\omega$, while the control of $D_v$ requires an additional estimate involving the primitive function $\Psi$. 
We now provide the details.

\begin{proof}[Proof of Proposition \ref{prop:1dlc:0} (i)]
Fix $x,y\in {\rm supp}(\rho_0)$ with $x>y$, and set
\[
X(t):= \eta(t,x)-\eta(t,y),\quad Y(t):= \omega(t,x) - \omega(t,y).
\]
From \eqref{eq:safe}--\eqref{eq:facile} it follows that $(X,Y)$ satisfies the assumptions of Lemma \ref{lem:1dlc:basic}. In particular, if $x,y$ are chosen so that
\[
x-y = D_{\eta}(0),
\]
then we have
\[
X(t)=D_{\eta}(t),\quad Y(t)=D_{\omega}(t)
\]
thanks to \eqref{eq:safe} and Lemma \ref{lem:1dlc:basic} (i), respectively.

By Lemma \ref{lem:1dlc:basic} (ii) and (iii), we immediately obtain
\[
D_\eta(t) + D_\omega(t) \to 0 \quad \mbox{as } t \to \infty.
\]
Moreover, the dynamics admit an explicit uniform bound. Indeed, since $Y(t)>0$ by Lemma \ref{lem:1dlc:basic} (i), we compute
\[
\frac{d}{dt} \lt(\lambda X(t)^2 + Y(t)^2\rt) \le 0,
\]
which yields
\[
D_\eta(t) \le \frac{1}{\lambda^{1/2}} \lt( \lambda D_\eta(0)^2 + D_\omega(0)^2 \rt)^{1/2} \quad \text{for all}\quad t\ge0.
\]
Finally, to control $D_v(t)$ we recall the definition of $\omega$ \eqref{def:om} and the upper bound \eqref{eq:Psi:ub}. For arbitrary $x,y\in {\rm supp}(\rho_0)$ with $x>y$ we have
\[
|v(x) -v(y)| \le |\omega(x)-\omega(y)| + 2\Psi\lt( \frac{\eta(x)-\eta(y)}{2}\rt).
\]
Taking the supremum over such pairs yields
\bq\label{D_v:estimate}
D_v(t) \le D_\omega(t) + 2\Psi\lt(\frac{D_\eta(t)}{2}\rt).
\eq
In particular, this inequality is of the form
\[
\dot X \geq Y - f(X), \quad f(r)=2\Psi\lt(\frac r2\rt).
\]
The function $f$ is continuous, increasing on $[0,\infty)$, and satisfies $f(0)=0$, thus the assumptions of Lemma \ref{lem:1dlc:basic} are satisfied. Consequently, Lemma \ref{lem:1dlc:basic} (iii) ensures that $Y(t)\to 0$, and together with \eqref{D_v:estimate} we deduce $D_v(t)\to 0$. Combining with the convergence of $D_\eta$ and $D_\omega$ concludes the proof.
 \end{proof}

%
%
%
%
%
%
%
%
%

\subsection{Coulomb--quadratic potential case}

We next turn to part (ii) of Proposition \ref{prop:1dlc:0}, where the interaction potential is given by \eqref{def_W_cc}. Unlike the $\lambda$-convex case, this potential is not convex, and thus contraction estimates cannot be applied directly. Instead, we analyze the coupled dynamics of characteristic separations and velocity differences through a suitable Lyapunov functional, which allows us to establish a uniform-in-time bound on $D_\eta(t)$.

We exploit the special structure of the Coulomb--quadratic potential, which reduces the dynamics to a system of two variables $(X,Y)$ similar to the $\lambda$-convex case but with an additional shift parameter $r>0$.  
The following auxiliary lemma provides the necessary uniform bound on $X(t)$, which in our application corresponds to the spatial diameter $D_\eta(t)$. 

\begin{lemma}\label{lem:aux:flock}
Let $(X,Y) \in C^1([0,\infty);[0,\infty) \times \R)$ satisfy
\bq\label{eq:auxsys}
\begin{split}
Y(t) - f(X(t)) \le &\dot{X}(t) \le Y(t), \\
&\dot{Y}(t) =-(X(t)-r)
\end{split}
\eq
for some $r>0$ and an increasing continuous function $f:[0,\infty) \to [0,\infty)$.
Then we have
\[
X(t) \le r + \max\lt\{ \sqrt{(X_0-r)^2+Y_0^2},\,\, \max\{Y_0,f(r)+1\} + \frac{r^2}{2} \rt\}
\]
for all $t\ge0$.
\end{lemma}
\begin{proof}
We introduce the Lyapunov-type functional
\[
\calL(t):= \frac{1}{2}\lt((X-r)^2+Y^2\rt).
\]
Differentiating and using \eqref{eq:auxsys} gives
\[
\dot\calL(t)=(X(t)-r)\dot X(t)+Y(t)\dot Y(t)
=(X(t)-r)\lt(\dot X(t)-Y(t)\rt).
\]
Since $Y-f(X)\le \dot X\le Y$, we have $- f(X)\le \dot X-Y\le 0$, hence
\begin{equation}\label{eq:dL-two-sided}
- f(X(t)) (X(t)-r)_+ \le \dot\calL(t) \le f(X(t)) (r-X(t))_+,
\end{equation}
where, for a real number $a$, $a_+ := \max\{a,0\}$ denotes the positive part of $a$.

\medskip
\noindent {\it Step 1: the case $X(s)\ge r$ for all $s\in[0,t]$.} From \eqref{eq:dL-two-sided} we obtain $\dot{\mathcal L}(s)\le 0$ on $[0,t]$, hence
\[
\calL(t)\le \calL(0)=\frac{1}{2}\lt((X_0-r)^2 + Y_0^2\rt).
\]

\medskip
\noindent {\it Step 2: the case $X$ crosses $r$ on $[0,t]$.}
Let
\[
t_p:=\sup\{s\in[0,t]: X(s)=r\}
\]
be the last time before $t$ at which $X$ hits the line $X=r$. By continuity, on the interval $[t_p,t]$ the trajectory stays entirely on one side of $X=r$; namely, either $X(s)\ge r$ for all $s\in[t_p,t]$ or $X(s)\le r$ for all $s\in[t_p,t]$.

\smallskip
\noindent {\it Case A: $X(s)\ge r$ on $[t_p,t]$.}
By \eqref{eq:dL-two-sided}, $\dot{\mathcal L}(s)\le 0$ on $[t_p,t]$, and thus
\[
\calL(t) \le \calL(t_p) = \frac{1}{2}Y(t_p)^2.
\]
We now bound $Y(t_p)$. At $t_p$, we get $X(t_p)=r$. From the inequalities $Y-f(X)\le \dot X \le Y$, we obtain
\[
Y(t_p)-f(r) \le \dot X(t_p)  \le Y(t_p).
\]
Since $t_p$ is the last hitting time of the line $X=r$, we have $\dot X(t_p)\ge 0$, which in turn yields $Y(t_p)\ge 0$. 
If $Y(t_p)\le f(r)+1$, the desired bound follows immediately. Otherwise, assume $Y(t_p)>f(r)+1$. In this case, we introduce
\[
\scrB:= \{s\in[0,t_p]: Y(s) \le f(r)+1\},\quad
t_l:=
\begin{cases}
\sup \scrB,& \scrB\neq\emptyset,\\ 0,& \scrB=\emptyset.
\end{cases}
\]
By construction, we find
\bq\label{eq:yts}
Y(t_l) = \begin{cases} f(r)+1 &\text{if}\,\,\, \scrB \neq \emptyset,\\
Y_0 > f(r)+1&\text{if}\,\,\, \scrB = \emptyset.
\end{cases}
\eq
 
 \medskip

\noindent {\it Claim.} On the interval $[t_l,t_p]$ the trajectory never rises above the line $X=r$, in other words,
\[
X(s)\le r \quad \text{for all } s\in[t_l,t_p].
\]

 \medskip
\noindent {\it Proof of claim.}
If not, there exists $s_0\in[t_l,t_p]$ such that $X(s_0)>r$. We now set
\[
s_1:=\inf\{\,s\ge s_0:\ X(s)=r\,\}.
\]
Then it is clear that $s_1<t_p$. Let $\tau\in(s_1,t_p)$ be a point where $X$ attains its minimum on $[s_1,t_p]$. Then $\dot X(\tau)=0$, and thus
\[
Y(\tau)\le \dot X(\tau)+f(X(\tau))\le f(r).
\]
Since $Y(t_p)>f(r)+1$, continuity yields $\tau_*\in(\tau,t_p)$ with $Y(\tau_*)=f(r)+1$, contradicting the definition of $t_l$. This proves the claim.

On $[t_l,t_p]$, we have $Y\ge f(r)+1$ by definition of $t_l$ and $X\le r$ by the claim, thus
\[
\dot{X} \ge Y - f(X) \ge Y - f(r) \ge 1.
\]
Hence
\[
\frac{dY}{dX}=\frac{\dot Y}{\dot X}=\frac{-(X-r)}{\dot X} \le r-X\quad \text{on }[t_l,t_p],
\]
and integrating from $X(t_l)$ to $X(t_p)=r$ gives
\[
Y(t_p)-Y(t_l) \le  \int_{X(t_l)}^{r}(r-X)\,dX  \le  \int_{0}^{r}(r-X)\,dX  =  \frac{r^2}{2}.
\]
Since $Y(t_l)\le \max\{Y_0,f(r)+1\}$ due to \eqref{eq:yts}, we conclude
\[
Y(t_p) \le \max\{Y_0,f(r)+1\} + \frac{r^2}{2}.
\]
Therefore, we obtain
\[
\calL(t_p) = \frac{1}{2}Y(t_p)^2\le \frac{1}{2}\lt(\max\{Y_0, f(r)+1\} + \frac{r^2}{2}\rt)^2.
\]

\medskip
\noindent {\it Case B: $X(s)\le r$ on $[t_p,t]$.} In this case, we directly have $X(t)\le r$, thus the desired estimate
\[
X(t) \le  r+\max\lt\{ \sqrt{(X_0-r)^2+Y_0^2},\, \max\{Y_0,f(r)+1\}+\frac{r^2}{2}\rt\}
\]
is immediate since the right-hand side is greater than $r$.
 
 We finally combine the above estimates together with $\frac{1}{2}(X(t)-r)^2 \le \calL(t)$ to conclude the desired result.
\end{proof}

 We are now ready to establish a uniform-in-time bound on the diameter of the density support in the case of the repulsive Coulomb potential with quadratic confinement. 
The strategy is to track the dynamics of two extreme characteristics, whose separation $X(t)$ and velocity difference $Y(t)$ evolve according to an auxiliary system of the form treated in Lemma \ref{lem:aux:flock}. The nonlinear alignment term can be controlled using Lemma \ref{lem:Psi}, ensuring that the assumptions of Lemma \ref{lem:aux:flock} are satisfied with $r=2$. This will yield a uniform bound on $X(t)=D_\eta(t)$ for all times.

\begin{proof}[Proof of Proposition \ref{prop:1dlc:0} (ii)]
We choose $x,y \in {\rm supp}(\rho_0)$ such that $D_\eta(0) = x-y$ so that by construction 
\[
D_\eta(t) = \eta(t,x)-\eta(t,y), \quad t \geq 0.
\]
From the reformulated system, we find
\begin{align*}
\pa_t \eta(t,x) &= \omega(t,x) - \int_\R \Psi(\eta(t,x)-\eta(t,z))\rho_0(z) \,dz, \cr
\pa_t \omega(t,x) &= -\eta(t,x) +(\eta_c(t) +1), \cr
\pa_t \eta(t,y) &= \omega(t,y) - \int_\R \Psi(\eta(t,y)-\eta(t,z))\rho_0(z) \,dz, \cr
\pa_t \omega(t,y) &= -\eta(t,y) +(\eta_c(t) -1),
\end{align*}
where we used that for the rightmost/leftmost Lagrangian variables we get $\eta_r(x) = 1$ and $\eta_r(y) = -1$.

Setting
\[
X(t):= D_\eta(t)=\eta(t,x) -\eta(t,y), \quad Y(t):=\omega(t,x)- \omega(t,y),
\]
we derive
\begin{align*}
\dot{X}(t) &= Y(t) - \int_\R \lt(\Psi(\eta(t,x)-\eta(t,z)) - \Psi(\eta(t,y)-\eta(t,z))\rt)\rho_0(z) \,dz,\cr
\dot{Y}(t) &= - (X(t) - 2).
\end{align*}
 On the other hand, it follows from Lemma \ref{lem:Psi}
\[
 0\le \int_\R \lt(\Psi(\eta(t,x)-\eta(t,z)) - \Psi(\eta(t,y)-\eta(t,z))\rt)\rho_0(z) \,dz \le 2\Psi\lt(\frac{1}{2}X(t)\rt).
\]
Moreover, the order-preserving property of characteristics guarantees $X(t)>0$ for all $t\ge 0$.  Thus the pair $(X(t),Y(t))$ satisfies exactly the structural assumptions of Lemma \ref{lem:aux:flock} with 
\[
r=2 \quad \mbox{and} \quad f(r)=2\Psi\lt(\frac r2\rt). 
\]
Applying that lemma, we conclude that
\[
X(t)=D_\eta(t) \le 2 + \max\lt\{ \sqrt{(D_\eta(0)-2)^2+D_\omega(0)^2}, \ \max\{D_{\omega}(0),2\Psi(1)+1\}+2\rt\},
\]
for all $t\ge 0$, which is precisely the asserted uniform bound.
\end{proof}

%
%
%
%
%
%
%
%
%
\section{One-dimensional dynamics with $(\lambda,\Lambda)$-convex interaction potential}\label{sec:m1}
 
 This section is devoted to the proof of Theorem \ref{thm_1:new}, which concerns the long-time asymptotics of \eqref{main_sys} in the one-dimensional setting with an attractive $(\lambda,\Lambda)$-convex interaction potential. Our goal is to establish quantitative decay rates for the characteristic diameters, as summarized in the following proposition.

\begin{proposition}\label{thm_1}Let $d=1$ and let $(\eta,v)$ be a global solution to \eqref{main_sys2} with sufficient regularity.  Suppose that $W$ satisfies \eqref{eq:lc:1d:ub}, and $\rho_0$ is compactly supported.
\begin{itemize}
\item[(i)]  If the communication weight function $\psi$ satisfies \eqref{psi:type1} (bounded and strictly positive near the origin), then the flow and velocity diameters decay exponentially fast:
\[
e^{-\psi_{\rm M} t} \lesssim D_\eta(t) + D_\omega(t) \lesssim e^{-\min\{\psi_{\rm m},\lambda/\psi_{\rm M}\} t},\quad D_v(t) \lesssim e^{-\min\{\psi_{\rm m},\lambda/\psi_{\rm M}\} t}
\]
for all $t\ge 0$.
\item[(ii)] If the communication weight $\psi$ satisfies \eqref{psi:type2} (weakly singular near the origin with exponent $\alpha \in (0,1)$), then the decay is only algebraic:
\[
D_\eta(t) = \Theta \lt(t^{-\frac{1}{\alpha}}\rt),\quad D_\omega(t)  = \Theta \lt(t^{-\frac{1}{\alpha} + 1}\rt), \quad  D_v(t) = \mathcal{O}  \lt(t^{-\frac{1}{\alpha}+1}\rt)
\]
as $t \to \infty$.
\end{itemize}
\end{proposition} 

\begin{remark}
In case (ii), the algebraic decay rates reflect the weaker alignment effect induced by the singular weight.  
While the spatial diameter $D_\eta(t)$ decays at the dominant rate $t^{-\frac1\alpha}$,  
the auxiliary variable $\omega$ and the velocity $v$ inherit slower convergence, 
decaying only at order $t^{-\frac1\alpha+1}$.  
This discrepancy highlights the difference between the contraction of trajectories 
and the relaxation of velocity fluctuations.
\end{remark}

To investigate quantitative decay estimates, we make use of the local assumptions on the communication weights \eqref{psi:type1} and \eqref{psi:type2}. 
In particular, Corollary \ref{cor:Psi} provides Lipschitz-type control of the primitive $\Psi$, which translates the alignment interaction into inequalities involving the pairwise distances of characteristics. 
Moreover, by Proposition \ref{prop:1dlc:0}, we already know that $D_\eta(t)\to 0$ as $t\to\infty$, and hence there exists $t_0>0$ such that
\[
D_\eta(t) \le R \quad \text{for all } t \ge t_0.
\]

In this case, Corollary \ref{cor:Psi}, combined with the one-dimensional relation \eqref{main_sys3}${}_2$ and the $(\lambda, \Lambda)$-convexity bound \eqref{eq:lc:1d:ub}, yields the differential inequality
\[
\pa_t(\omega(t,x)-\omega(t,y)) \ge -\Lambda(\eta(t,x)-\eta(t,y)).
\]
Recalling that $X(t)$ and $Y(t)$ denote the pairwise deviations 
\[
X(t) := \eta(t,x)-\eta(t,y), \quad Y(t) := \omega(t,x)-\omega(t,y),
\]
the inequality above reads
\[
\dot{Y}(t) \ge -\Lambda X(t).
\]
Collecting these estimates, we deduce that the dynamics of $(X,Y)$ are governed, for all $t\ge t_0$, by the following system of ODIs:
\bq\label{eq:1dlc:master}
\begin{aligned}
Y (t)- c_1 X^{1-\alpha}(t)\le &\dot{X}(t) \le Y(t) - c_2 X^{1-\alpha}(t),\\
-\Lambda X(t) \le &\dot{Y}(t) \le -\lambda X(t)
\end{aligned}
\eq
with some constants $c_1 \ge c_2 > 0$ and $\lambda >0$. Here the exponent $\alpha \in [0,1)$ encodes the local behavior of $\psi$ near the origin as specified in \eqref{psi:type1}--\eqref{psi:type2}. The system \eqref{eq:1dlc:master} will serve as the reduced model for analyzing the large-time dynamics in the one-dimensional case.

%
%
%
%
%
%
%
%
%
%
\subsection{Bounded communication weights}\label{sec:1lb}

We begin with the case of bounded communication weights, which corresponds to the ODI system \eqref{eq:1dlc:master} with $\alpha=0$. In this case, the alignment operator is Lipschitz continuous, and the interaction terms lead to exponential convergence. The precise estimate is stated as follows.

\begin{lemma}\label{lem:ODI_1}
Let $(X,Y)$ be a $C^1([0,\infty);[0,\infty) \times \R)$ curve satisfying  
\bq \label{ODI_1}
\begin{aligned}
Y(t) - c_1X(t) \le &\dot{X}(t) \le Y(t) - c_2 X(t), \\
-\Lambda X(t) \le &\dot{Y}(t) \le -\lambda X(t)
\end{aligned}
\eq
for some $\Lambda \ge \lambda >0$ and $c_1 \ge c_2 >0$. Then $Y(t)>0$ for all $t\ge0$, and we have the two-sided bound
\bq\label{exp:lbub}
C_1 e^{-c_1 t}  \leq  X(t) + Y(t) \leq C_2 e^{-\min\lt\{c_2,\lambda/c_1 \rt\}t}, \quad t \geq 0
\eq
for some positive constants $C_1, C_2$ independent of $t$.
\end{lemma}

\begin{remark}
The estimate \eqref{exp:lbub} provides both lower and upper exponential bounds for $X(t)+Y(t)$. In particular, the upper decay rate depends only on $c_2$ and $\lambda / c_1$, and is therefore sharper than the one obtained by standard hypocoercivity methods, where the presence of $\Lambda$ typically reduces the decay rate when $\Lambda$ is large.  
The key ingredient of the proof is that $X(t)\ge0$ holds for all $t\ge0$, which allows us to employ a phase-plane argument to derive an improved exponential decay.
\end{remark}

\begin{proof}[Proof of Lemma \ref{lem:ODI_1}]
We first recall from Lemma \ref{lem:1dlc:basic} (i) that $Y(t)>0$ for all $t\ge0$. With $X(t),Y(t)>0$, differentiating and using \eqref{ODI_1} yields
\[
\frac{d}{dt}\lt(\Lambda X^2 +Y^2\rt) \ge -2c_1\Lambda X^2 \ge -2c_1\lt(\Lambda X^2 +Y^2\rt),
\]
so that
\[
\Lambda X^2(t) + Y(t)^2 \ge \lt(\Lambda X_0^2 +Y_0^2\rt)e^{-2c_1 t},
\]
which establishes the lower bound in \eqref{exp:lbub}.

For the upper bound, observe  \eqref{ODI_1} implies that
\[
\frac{d}{dt}\lt(\lambda X^2 +Y^2\rt) \le -2 c_2\lambda X^2 \le 0,
\]
which immediately yields
\[
\lambda X^2(t) + Y^2(t) \le \lambda X_0^2 +Y_0^2.
\]
This bound, however, does not yet provide quantitative decay. To obtain a sharper decay estimate, we rewrite the inequality as
\bq\label{eq:folk:pre}
\frac{d}{dt}\lt(\lambda X^2 +Y^2\rt) \le -2 c_2 \lt(\lambda X^2 +Y^2\rt) + 2c_2 Y^2.
\eq
Hence, the key step is to obtain a quantitative decay estimate for $Y(t)$, and then close the estimate for $\lambda X^2 +Y^2$.

To this end, we claim that there exists $t_i\ge0$ such that 
\bq\label{Y:claim}
Y(t) \le c_1 X(t)\quad \text{for all }  t \ge t_i.
\eq
Suppose first that $Y(t_i) \le c_1 X(t_i)$ for some $t_i \ge 0$. Then, if $Y(t_*) = c_1 X(t_*)$  for some $t_* \ge t_i$, we have
\[
\dot{X}(t_*) \ge 0\quad \mbox{while } \dot{Y}(t_*) <0,
\]
due to \eqref{ODI_1}. This provides that $Y(t_* +\e) < Y(t_*)= c_1 (X_*) \le c_1 X(t_*+\e)$ for all sufficiently small $\e>0$. This shows that once the relation holds, it remains true thereafter.

It remains to prove that such a finite $t_i$ exists. Define
\[
t_i:= \inf \lt\{t \ge 0: Y(t) \le c_1 X(t) \rt\}.
\]
If $Y_0 \le c_1 X_0$, then $t_i=0$. Otherwise, for $t \in [0,t_i)$, we have $\dot{X}(t) \ge 0$, and hence
\[
\dot{Y}(t) \le - \lambda X_0.
\]
Thus $Y$ decreases at least linearly until the threshold $Y = c_1 X$ is reached, and consequently $t_i<\infty$. In particular, it follows from 
\[
c_1 X_0 \le c_1 X(t_i) = Y(t_i) \le Y_0 - \lambda X_0 t_i,
\]
that
\[
t_i \le \frac{1}{\lambda}\lt(\frac{Y_0}{X_0} - c_1 \rt)_+.
\]
Having established the claim \eqref{Y:claim}, we can combine it with 
$\dot{Y}\le -\lambda X$ to obtain
\[
\dot{Y}(t) \le -\frac{\lambda}{c_1}Y(t)\quad \text{for all }  t \ge t_i.
\]
By Gr{\"o}nwall's inequality, we then find
\[
Y(t) \le Y(t_i)e^{-\frac{\lambda}{c_1}(t-t_i)} \quad \text{for all }  t \ge t_i,
\]
and since $Y$ is monotone decreasing, this further implies
\[
Y(t)   \le Y_0 e^{\frac{Y_0}{c_1 X_0}} e^{-\frac{\lambda}{c_1}t} \quad \text{for all }  t \ge 0.
\]
At this point, the decay of $Y$ allows us to return to \eqref{eq:folk:pre} and apply Gr\"onwall's lemma again to deduce
\[
\lambda X(t)^2 + Y(t)^2 \le \lt(\lambda X_0^2 +Y_0^2\rt)e^{-2c_2t} + 2c_2 Y_0^2e^{\frac{2Y_0}{c_1X_0}}\lt(\frac{1}{c_2} + \frac{c_1}{\Lambda}\rt) e^{- 2\mathfrak{h}(c_2,\lambda/c_1)t}
\]
for all $t\ge 0$, where
\[
\mathfrak{h}\lt(c_2,\frac\lambda{c_1} \rt):=\frac{1}{\frac{1}{c_2}+\frac{c_1}{\lambda}}= \frac{c_2 \lambda}{\lambda + c_1 c_2}
\]
is the harmonic mean of $c_2$ and $\frac{\lambda}{c_1}$. This establishes the desired upper bound and completes the proof of Lemma \ref{lem:ODI_1}. 
\end{proof}

We now prove Proposition \ref{thm_1} (i), concerning the bounded communication case.

 \begin{proof}[Proof of Proposition \ref{thm_1} (i)]
By Proposition \ref{prop:1dlc:0} and Corollary \ref{cor:Psi} (i), for all sufficiently large times $t\ge t_0$ the pair
\[
X(t):=D_\eta(t), \quad Y(t):=D_\omega(t)
\]
satisfies the ODI system \eqref{ODI_1} with the choices
\[
c_1=\psi_{\rm M},\quad c_2=\psi_{\rm m},\quad (\lambda, \Lambda)=\text{constants in \eqref{eq:lc:1d:ub}}.
\]
Applying Lemma \ref{lem:ODI_1} then yields the two-sided exponential bounds
\[
e^{-\psi_{\rm M} t} \lesssim  X(t)+Y(t) \lesssim  e^{-\min\{\psi_{\rm m}, \frac\lambda{\psi_{\rm M}}\} t}, \quad t\ge t_0.
\]
Since $t_0$ is fixed, this implies the stated estimate for $D_\eta+D_\omega$ on $[0,\infty)$ after adjusting constants.

Finally, using \eqref{D_v:estimate} together with Corollary \ref{cor:Psi} (i) gives
\[
D_v(t) \lesssim D_\omega(t)+ D_\eta(t),
\]
hence $D_v(t)$ decays with the same upper rate $e^{-\min\lt\{\psi_{\rm m}, \frac\lambda{\psi_{\rm M}}\rt\} t}$. This completes the proof.
\end{proof}

%
%
%
%
%
%
%
%
%
%

\subsection{Weakly singular communication weights}\label{sec:1dws}
 In Lemma \ref{lem:ODI_1}, we exploited the interplay between trajectories and velocity differences to derive exponential decay in the bounded case. A similar trajectory-based argument will also be effective for the weakly singular setting $\alpha \in (0,1)$, although the analysis becomes more delicate due to the lack of uniform coercivity. We first set up a general framework that reduces the dynamics to a one-dimensional relation between $X$ and $Y$.
 
 Let $(X(t),Y(t))$ be a $C^1([0,\infty); (0,\infty) \times \R)$ solution to \eqref{eq:1dlc:master}. From Lemma \ref{lem:1dlc:basic}, we know that $Y(t)$ is strictly decreasing and tends to $0$ as $t\to\infty$. Thus $Y:[0,\infty)\to(0,Y_0]$ is invertible, and we may reparametrize $X$ in terms of $Y$. More precisely, if $\tau:(0,Y_0]\to[0,\infty)$ denotes the inverse of $Y$, we set
 \[
 g(y) := X(\tau(y)), \quad y \in (0,Y_0], \quad g(0):=0.
 \]
 By construction, 
 \bq\label{eq:connect}
X(t)=g(Y(t))\quad \text{for all } t\ge0.
\eq
 Moreover, since $(X,Y)$ is $C^1$ in time, the function $g$ is continuous on $[0,Y_0]$ and differentiable on $(0,Y_0]$, with $g(y)>0$ whenever $y>0$. This reparametrization allows us to reduce the two-dimensional dynamics of $(X,Y)$ to a one-dimensional relation between $X$ and $Y$, which is a standard device in the phase-plane method.

From the second line of \eqref{eq:1dlc:master} we have $\dot{Y}\le -\lambda X=-\lambda g(Y)$, and thus,
\[
\frac{dY(t)}{\lambda  g(Y(t))} \le -dt.
\]
For any fixed $t_i \ge 0$, we integrate both sides in $[t_i,t]$ to obtain
\bq\label{clock:ub}
t - t_i \le \frac{1}{\lambda}\int_{Y(t)}^{Y(t_i)} \frac{dy}{g(y)}.
\eq
A symmetric manipulation of the lower bound $\dot{Y}\ge -\Lambda X=-\Lambda g(Y)$ gives
\bq\label{clock:lb}
 \frac{1}{\Lambda}\int_{Y(t)}^{Y(t_i)} \frac{dy}{g(y)} \le t - t_i.
\eq
These two relations will provide upper and lower bounds on $Y(t)$, respectively. The bounds on $X(t)$ will then follow from \eqref{eq:connect}.

To derive quantitative bounds, we analyze the asymptotic behavior of $g$ near $0$. This can be achieved by examining the differential inequalities satisfied by $g$. Differentiating the relation \eqref{eq:connect} yields
\[
\dot{X}(t) = g'(Y(t))\dot{Y}(t).
\]
Since $-\Lambda X \le \dot{Y} \le -\lambda X$, the sign of $\dot{X}(t)$ and $\dot{Y}(t)$ must be carefully identified to determine the direction of the inequality for $g'$. Note in particular that $\dot{Y}<0$, while the sign of $\dot{X}$ can be characterized in specific regimes.
Indeed, we can readily find that
\bq\label{eq:g1}
g'(y) \le \frac{c_1 g^{1-\alpha}(y)- y}{\lambda g(y)}\quad \text{if }  g^{1-\alpha}(y) \le \frac{y}{c_1},
\eq
and
\bq\label{eq:g2}
g'(y) \ge \frac{c_2 g^{1-\alpha}(y)- y}{\Lambda g(y)}\quad \text{if }  g^{1-\alpha}(y) \ge \frac{y}{c_2}.
\eq
These inequalities describe the local trajectory of $(X,Y)$ in the phase plane and will play a crucial role in capturing the large-time asymptotics.

With these preparations in place, we can now establish the following lemma, which provides the precise algebraic decay of $(X,Y)$ for weakly singular communication weights.

\begin{lemma}\label{lem:ODI_2}
Let $(X,Y)$ be a $C^1([0,\infty);(0,\infty) \times \R)$ curve satisfying
\bq \label{ODI_2}
\begin{aligned}
Y(t) - c_1X^{1-\alpha}(t) \le &\dot{X}(t) \le Y(t) - c_2X^{1-\alpha}(t), \\
-\Lambda X(t) \le &\dot{Y}(t) \le -\lambda X(t), 
\end{aligned}
\eq
where $\alpha \in (0,1)$ with $\Lambda \ge \lambda >0$ and $c_1 \ge c_2 >0$. Then we have
\[
X(t) = \Theta \lt(t^{-\frac{1}{\alpha}}\rt), \quad Y(t) = \Theta \lt(t^{-\frac{1}{\alpha}+1 }\rt)
\]
as $t \to \infty$.
\end{lemma}

\begin{proof}
Let us denote a function $g:[0,Y_0] \to [0,\infty)$ as defined in \eqref{eq:connect}. We introduce
\[
h(y):= g^{1-\alpha}(y).
\]
This change of variables is natural because the exponent $1-\alpha$ precisely balances the nonlinearity in \eqref{ODI_2}.
In terms of $h$, the differential inequalities for $g$ in \eqref{eq:g1}--\eqref{eq:g2} become
\bq\label{eq:h1}
 h'(y) \le \frac{(1-\alpha)(c_1 h(y)- y)}{\lambda h^{\frac{1+\alpha}{1-\alpha}}(y)}\quad \text{if }  h(y) \le \frac{y}{c_1},
\eq
and
\bq\label{eq:h2}
 h'(y) \ge \frac{(1-\alpha)(c_2 h(y)- y)}{\Lambda h^{\frac{1+\alpha}{1-\alpha}}(y)}\quad \text{if }  h(y) \ge \frac{y}{c_2}.
\eq
Our goal is to prove that there exists $y_i>0$ such that
\bq\label{eq:h:goal}
\frac{y}{c_1} \le h(y) \le \frac{2y}{c_2} \quad \text{for all }  y\in[0, y_i].
\eq
This relation captures the local linear behavior of $h$ near the origin, and consequently allows us to determine the decay rates of $X$ and $Y$.

\medskip
\noindent {\it Step 1: Establishing the lower bound for $h$.} We first show that there exists $0<y_1\le Y_0$ such that
\bq\label{eq:step1}
h(y) \ge \frac{y}{c_1}\quad \text{for all }  y\in [0,y_1].
\eq

\smallskip
\noindent  {\it Step 1(a): Persistence once the barrier is crossed.} Suppose that for some $y_1\in(0,Y_0]$ we have $h(y_1)\ge \frac{y_1}{c_1}$. We claim that then \eqref{eq:step1} holds for all smaller $y$. Indeed, if not, then there exists $\tilde y \in (0,y_1)$ such that $h(\tilde y) < \frac{\tilde y}{c_1}$. Let
\[
y_*:=\sup\lt\{[0,y_1]: h(y) < \frac{y}{c_1}\rt\}.
\]
 We then observe that 
 \[
 h(y_*) = \frac{y_*}{c_1} \quad \mbox{and} \quad h'(y_*) \ge \frac{1}{c_1}.
 \]
On the other hand, inequality \eqref{eq:h1} gives $h'(y_*)\le0$, which is a contradiction.
Thus, the lower barrier $y/c_1$ is preserved once it is reached.

\smallskip
\noindent {\it Step 1(b): Existence of the crossing point.} It remains to show that such a point $y_1$ exists. Suppose to the contrary that $h(y)<\frac{y}{c_1}$ for all $y\in(0,Y_0]$. Then by \eqref{eq:h1} we have $h'(y)\le0$ on $(0,Y_0]$, so $h$ is nonincreasing. This is incompatible with $h(0)=0$ and $h>0$ on $(0,Y_0]$.
To make this precise, set $y_1:=c_1 h(y_0)<Y_0$. Then $h(y_1)\ge h(Y_0)=\frac{y_1}{c_1}$, and thus this leads to a contradiction. Hence, a crossing point $y_1$ does exist, and \eqref{eq:step1} follows.

Step 1 thus establishes the desired lower control of $h$ near zero.
This is directly analogous to the barrier argument used in Lemma \ref{lem:ODI_1}, and it ensures that $h$ grows at least linearly with slope $1/c_1$.

In the next step, we turn to the complementary upper bound, which is subtler.
Unlike the lower estimate, the upper bound requires exploiting the positivity of $\alpha$ in order to prevent $h$ from exceeding a linear growth rate.

\medskip
\noindent {\it Step 2: Establishing the upper bound for $h$.} We now show that $h$ cannot stay strictly above the barrier $\frac{2y}{c_2}$ near the origin. More precisely, there exists $0<y_2\le Y_0$ such that
\bq\label{eq:step3}
h(y) \le \frac{2y}{c_2}\quad \text{for all }  y\in [0,y_2].
\eq
To prepare for the proof, we first fix a cutoff scale $y_0\in(0,Y_0]$ small enough to guarantee
\bq\label{dom:cutoff}
\sup_{y \in [0,y_0]}h^{\frac{2\alpha}{1-\alpha}}(y) = \sup_{y \in [0,y_0]}g^{2\alpha}(y) < \frac{c_2^2 (1-\alpha)}{4\Lambda}.
\eq
The existence of such a $y_0$ follows from the continuity of $g$ and the fact that $g(0)=0$.
Within this interval $[0,y_0]$, we now establish the upper bound in two substeps.
 
\smallskip
\noindent {\it Step 2(a): Persistence once the barrier is crossed.} If $h(y_2) \le \frac{2y_2}{c_2}$ for some $0<y_2\le y_0$, then \eqref{eq:step3} holds automatically on $[0,y_2]$. Indeed, suppose instead that $h(\tilde y) > \frac{2\tilde y}{c_2}$ for some $\tilde y \in (0, y_2)$, and define 
\[
y^*:=\sup\lt\{[0,y_1]: h(y) > \frac{2y}{c_2}\rt\}.
\]
Then, we find 
\[
h(y^*) = \frac{2y^*}{c_2} \quad \mbox{and} \quad h'(y^*) \le \frac{2}{c_2}. 
\]
However, applying \eqref{eq:h2} at $y^*$ yields
\[
h'(y^*) \ge  \frac{(1-\alpha)(c_2 h(y^*)- y^*)}{\Lambda h^{\frac{1+\alpha}{1-\alpha}}(y^*)}=  \frac{(1-\alpha)c_2}{2\Lambda h^{\frac{2\alpha}{1-\alpha}}(y^*)} > \frac{2}{c_2},
\]
where the last inequality follows from \eqref{dom:cutoff}. This contradiction shows that once $h$ touches the barrier, it cannot cross above it again.

\smallskip
\noindent {\it Step 2(b): Existence of a crossing point.}  It remains to prove that $h$ does meet the barrier at some finite $y_2$. Suppose to the contrary that $h(y) > \frac{2y}{c_2}$ for all $y \in(0,y_0]$. By rearranging the inequality \eqref{eq:h2}, we deduce that
\[
h^{\frac{2\alpha}{1-\alpha}}(y)h'(y) \ge \frac{(1-\alpha)\lt(c_2 - \frac{1}{\frac{h(y)}{y}}\rt)}{\Lambda}\ge \frac{(1-\alpha)c_2}{2\Lambda},
\]
in other words,
\[
(h^{\frac{1+\alpha}{1-\alpha}})'(y) \ge \frac{(1+\alpha)c_2}{2\Lambda} =: \e >0.
\]
Integrating from $0$ with $h(0)=0$ gives
\[
h^{\frac{1+\alpha}{1-\alpha}}(y) \ge \e y, \quad y \in [0,y_0].
\]
Thus, we get
\[
g(y) = h^{\frac{1}{1-\alpha}}(y)\ge (\e y)^{\frac{1}{1+\alpha}}, \quad y \in [0,y_0].
\]
We put this into \eqref{clock:ub} with sufficiently large $t_i$ satisfying $Y(t_i) \le y_0$. Then, for any $t \ge t_i$, we have
\[
t - t_i \le \frac{1}{\lambda}\int_{Y(t)}^{Y(t_i)} \frac{dy}{g(y)} \le  \frac{1}{\lambda}\int_{0}^{y_0} \frac{dy}{(\e y)^{\frac{1}{1+\alpha}}} = \frac{(1+\alpha)}{\alpha \e^{\frac{1}{1+\alpha}}}y_0^{\frac{\alpha}{1+\alpha}} <+\infty,
\]
which contradicts the fact that $t\to\infty$ while $Y(t)\to0$. Hence, $h$ must cross the barrier at some finite $y_2\le y_0$.

\medskip
\noindent {\it Step 3: Deriving the optimal decay rates for $X(t)$ and $Y(t)$.} Combining the results of {\it Steps 1 \& 2}, we select
\[
y_i:= \min\{y_1,y_2\}>0, 
\]
so that the two-sided estimate \eqref{eq:h:goal} holds for all $y \in [0,y_i]$. We now exploit this local behavior of $h$ near the origin to deduce the asymptotic time decay of $Y(t)$ and $X(t)$. 

Take a sufficiently large time $t_i \ge 0$ such that $Y(t_i) = y_i$. Applying the inequality \eqref{clock:ub} together with the bound \eqref{eq:h:goal} and the relation $g(y) = h^{\frac{1}{1-\alpha}}(y)$, we obtain, for all $t \ge t_i$,
\[
t - t_i \le \frac{1}{\lambda}\int_{Y(t)}^{y_i} \frac{dy}{h^{\frac{1}{1-\alpha}}(y)} \le \frac{c_1^{\frac{1}{1-\alpha}}}{\lambda}\int_{Y(t)}^{y_i} \frac{dy}{y^{\frac{1}{1-\alpha}}}. 
\]
Evaluating the integral on the right-hand side yields
\[
t - t_{i} \le \frac{c_1^{\frac{1}{1-\alpha}}}{\lambda}\lt(\frac{1-\alpha}{\alpha}\rt)\lt(Y(t)^{-\frac{\alpha}{1-\alpha}}- {y_i}^{-\frac{\alpha}{1-\alpha}}\rt),
\]
and thus
\[
Y(t) \le \lt(\lt(\frac{\lambda \alpha}{c_1^{\frac{1}{1-\alpha}}(1-\alpha)}\rt)(t-t_i) + y_i^{-\frac{\alpha}{1-\alpha}}\rt)^{-\frac{1-\alpha}{\alpha}}.
\]
A similar argument using \eqref{clock:lb} and the bound \eqref{eq:h:goal} provides the lower bound estimate
\[
Y(t) \ge \lt(\lt(\frac{2^{\frac{1}{1-\alpha}}\Lambda \alpha}{c_2^{\frac{1}{1-\alpha}}(1-\alpha)}\rt)(t-t_i) + y_i^{-\frac{\alpha}{1-\alpha}}\rt)^{-\frac{1-\alpha}{\alpha}}.
\]
Together, these bounds establish the precise algebraic decay rate of $Y(t)$.
Finally, recalling the relation $X(t)=g(Y(t))$ and the bounds in \eqref{eq:h:goal}, we obtain the corresponding decay estimate for $X(t)$, thereby completing the proof.
\end{proof}

Based on Lemmas \ref{lem:ODI_1} and \ref{lem:ODI_2}, we provide the proof of Proposition \ref{thm_1} (ii).

\begin{proof}[Proof of Proposition \ref{thm_1} (ii)]
Let $X(t),Y(t)$ be as defined above. By Proposition \ref{prop:1dlc:0}, we know that there exists $t_0>0$ such that $D_\eta(t)\le R$ for all $t \ge t_0$. This allows us to apply Lemma \ref{lem:ODI_2}, which yields the precise algebraic decay of $D_\eta(t)$ and $D_\omega(t)$ in the case $\alpha \in (0,1)$. Finally, recalling \eqref{D_v:estimate} together with Corollary \ref{cor:Psi}, we obtain the corresponding decay estimate for $D_v(t)$.
This establishes Proposition \ref{thm_1} (ii).
\end{proof}

%
%
%
%
%
%
%
%
%
%
 \subsection{Proof of Theorem \ref{thm_1:new}}
We now translate the decay estimates obtained in the Lagrangian framework back to the original Eulerian coordinates in order to complete the proof of Theorem \ref{thm_1:new}. Recall that the asymptotic profile is represented by the moving Dirac mass $\rho_\infty(t) = \delta_{\eta_c(t)}$, where the center of mass trajectory $\eta_c(t)$ lies in the convex hull of the current support of $\rho(t)$. More precisely,
\[
\eta_\infty(t,x) = \eta_c(t) = \int_\R \eta(t,x)\rho_0(x)\,dx \in \text{conv}\lt\{\eta(t,x): x \in {\rm supp}(\rho_0)\rt\}.
\]
Consequently, the diameter of the relative displacement satisfies
\[
D_{\tilde{\eta}}(t) \le D_{\eta}(t).
\]
Combining this relation with Lemma \ref{lem:Wst} yields the upper bound
\[
 {\rm d}_\infty (\rho(t), \rho_\infty(t)) \le D_{\tilde{\eta}}(t) \le D_\eta(t).
\]
On the other hand, since $\rho_\infty(t)$ is concentrated at the center of mass, any coupling between $\rho(t)$ and $\rho_\infty(t)$ must transport at least half of the mass across a distance comparable to the support diameter of $\rho(t)$. This provides the complementary lower bound
\[
 {\rm d}_\infty(\rho(t),\rho_\infty(t)) \ge \frac{1}{2}D_\eta(t).
\]
Together, these two inequalities show that the asymptotic behavior of the density in Wasserstein distance is quantitatively equivalent to the decay of the flow diameter $D_\eta(t)$.

For the velocity field, we exploit the fact that the limiting velocity $u_\infty$ belongs to the convex hull of the instantaneous velocities along characteristics:
\begin{align*}
u_\infty &= \int_\R u_0(y)\rho_0(y)\,dy = \int_\R u(t,y)\rho(t,y)\,dy \cr
&= \int_\R u(t,\eta(t,y)) \rho_0(y)\,dy \in \text{conv}\lt\{u(t,\eta(t,y)):y \in {\rm supp}(\rho_0)\rt\}.
\end{align*}
Therefore, for any $x \in {\rm supp}(\rho_0)$ we obtain
\[
|u(t,\eta(t,x)) - u_\infty| \le D_v(t),
\]
and hence the uniform estimate
\[
\|u(t,\cdot) - u_\infty\|_{L^\infty({\rm supp}(\rho(t)))} \le D_v(t).
\]

Putting these Eulerian interpretations together, Proposition \ref{thm_1} ensures the decay of both $D_\eta(t)$ and $D_v(t)$ with the corresponding exponential or algebraic rates depending on the communication weights. As a consequence, the Wasserstein distance between $\rho(t)$ and the moving Dirac mass $\rho_\infty(t)$ converges to zero, and the velocity field aligns uniformly to $u_\infty$. This proves the desired stability estimates in Theorem \ref{thm_1:new}.

%
%
%
%
%
%
%
%
%
%
\section{One-dimensional dynamics with Coulomb--quadratic potential}\label{sec:m2}

In this section, we investigate the large-time behavior in the case of a repulsive Coulomb potential with quadratic confinement \eqref{def_W_cc}. This case is not covered by the $(\lambda,\Lambda)$-convex framework and requires a more delicate analysis. To this end, we introduce perturbation variables around the reference profile $(\eta_\infty,u_\infty)$, which allow us to quantify deviations of trajectories and velocities from the asymptotic state.

We recall the perturbed system:
\[
\begin{aligned}
\pa_t \tilde \eta(x)  & = \tilde\omega(x)  - \int_\R (\Psi(\eta(x)-\eta(y)) - \Psi(\eta_\infty(x)-\eta_\infty(y)))\rho_0(y)\,dy - u_\infty, \cr
\pa_t \tilde \omega(x)   & = - \tilde \eta(x),
\end{aligned}
\]
where 
\[
\tilde \eta = \eta - \eta_\infty, \quad \tilde \omega(x) = v(x) + \int_\R (\Psi(\eta(x)-\eta(y)) - \Psi(\eta_\infty(x)-\eta_\infty(y)))\rho_0(y)\,dy.
\]
It is worth noticing that $\tilde \eta$ does not have a sign preserving property \eqref{sgn}. This lack of monotonicity makes the analysis substantially more involved.

To quantify the contraction of trajectories and velocity discrepancies in this perturbative setting, we introduce the diameters for $\tilde\eta$ and  $\tilde \omega$. Analogous to the previous definition, we set
\[
D_{\tilde\eta}(t) := \sup_{x,y \in {\rm supp}(\rho_0)}|\tilde\eta(t,x) - \tilde\eta(t,y)|, \quad D_{\tilde\omega}(t) := \sup_{x,y \in {\rm supp}(\rho_0)}|\tilde\omega(t,x) - \tilde\omega(t,y)|.
\]
These functionals measure the maximal deviations from the asymptotic profile $\tilde \eta_\infty$ in position and in the modified velocity, respectively, and their decay properties will play an important role in establishing convergence toward flocking states.

The goal of this section is to establish the quantitative decay of these diameters and related fluctuation energies under suitable assumptions on the communication weight $\psi$. Specifically, we prove that compactly supported solutions converge toward flocking states, with exponential decay in the case of bounded weights and algebraic decay in the case of weakly singular weights. This is summarized in the following proposition.

\begin{proposition}\label{thm_2}Let $d=1$ and let $(\eta,v)$ be a global solution to \eqref{main_sys2} with compactly supported initial density $\rho_0$ and sufficient regularity. 
Suppose that the assumptions of Proposition \ref{prop:1dlc:0} are satisfied, which ensures $D_\eta(t)\le \overline{D}$ for all $t\ge0$. Assume further that the communication weight satisfies
\[
\psi(x) \ge \psi_{\rm m}:= \psi(\bar D)>0 \quad \mbox{for all } |x| \le \bar D.
\]
Then the following decay estimates hold:
\begin{itemize}
\item[(i)] If $\psi$ is bounded from above, then the relative diameters and the velocity diameter decay exponentially:
\[
D_{\tilde{\eta}}(t) + D_{\tilde{\omega}}(t) + D_v(t) \to 0 \quad \mbox{as } t \to \infty.
\]
\item[(ii)]  If $\psi$ satisfies $\psi(x) \le B|x|^{-\alpha}$ for some $\alpha \in (0,1)$ and all $|x| \le \bar D$, then the decay is algebraic:
\[
\iint_{\R \times \R} \lt((\tilde \eta(t,x) - \tilde \eta(t,y))^2 + (\tilde\omega(t,x) -\tilde\omega(t,y))^2 \rt)\rho_0(x)\rho_0(y)\,dxdy = \mathcal{O}\lt(t^{-(\frac{3}{2\alpha}-1) }\rt)
\]
and
\[
\iint_{\R \times \R} (v(t,x) - v(t,y))^2 \rho_0(x)\rho_0(y)\,dxdy = \mathcal{O} \lt(t^{-(1-\alpha)(\frac{3}{2\alpha}-1) }\rt)
\]
as $t \to \infty$.
\end{itemize}
\end{proposition}

\begin{remark}
When the communication weight is constant, the long-time asymptotics can be obtained without uniform control of the diameter, since the system admits direct coercivity estimates. In fact, if $\psi(\cdot)=\bar{\psi}>0$, then $\Psi(s)=\bar{\psi}s$ and the dynamics reduce to
$$\begin{aligned}
&\pa_t (\tilde \eta(x) - \tilde \eta(y) )= (\tilde\omega(x) - \tilde\omega(y)) - \bar{\psi}(\tilde \eta(x) - \tilde \eta(y)),\cr
&\pa_t (\tilde \omega(x) - \tilde \omega(y)) = - (\tilde \eta(x) - \tilde \eta(y)).
\end{aligned}$$
In this case, hypocoercivity arguments (similar to those in the proof of Proposition \ref{thm_1} for the bounded weight case in Section \ref{sec:1lb}) yield the exponential decay
\[
D_{\tilde \eta}(t)+D_{\tilde \omega}(t)\to 0 \quad \text{as }t\to\infty.
\]
Furthermore, the pointwise estimate
\[
|v(x)-v(y)| \le |\tilde \omega(x) - \tilde \omega (y)| + \bar{\psi} |\tilde \eta (x) - \tilde \eta (y)|
\]
implies that $D_v(t)$ decays exponentially as well. This recovers the sharp decay obtained in \cite{CCZ16} by explicit formulas for characteristics.
\end{remark}

\begin{remark}\label{rmk:relax}
In fact, the assumptions on the local behavior of the communication weight $\psi$ as in Theorem \ref{thm_2:new} hold only \emph{a priori} on $|x|\le R$ for some $R>0$. However, they imply that they hold for all $|x| \le \overline{D}$ with possibly different constants due to the fact that $\psi$ is a positive monotone decreasing function of $|x|$. 
\end{remark}

We next establish two fundamental dissipation estimates for the perturbation variables. These estimates, formulated in terms of $\tilde\eta$, $\tilde\omega$, and $v$, quantify the decay of relative fluctuations and will serve as the starting point for deriving both exponential and algebraic convergence rates in the subsequent analysis.

\begin{lemma}\label{lem:aux-energy}Let $d=1$ and let $(\eta,v)$ be a global solution to \eqref{main_sys2} with compactly supported initial density $\rho_0$ and sufficient regularity.  Suppose that the assumptions of Proposition \ref{thm_2} are satisfied. Then the following dissipation estimates hold:
\begin{align}\label{est_uw_l2_2}
\begin{aligned}
&\frac12\frac{d}{dt} \iint_{\R \times \R} \lt((\tilde \eta(x) - \tilde \eta(y))^2 + (\tilde\omega(x) - \tilde\omega(y))^2 \rt)\rho_0(x)\rho_0(y)\,dxdy\cr
&\quad \leq - \psi_{\rm m}\iint_{\R \times \R}  (\tilde \eta(x) - \tilde \eta(y))^2 \rho_0(x)\rho_0(y)\,dxdy
\end{aligned}
\end{align}
and
\begin{align}\label{est_uw_l2_4}
\begin{aligned}
&\frac12\frac{d}{dt}\iint_{\R \times \R} \lt((\tilde \eta(x) - \tilde \eta(y))^2 + (v(x) - v(y))^2 \rt)\rho_0(x)\rho_0(y)\,dxdy\cr
&\quad\leq -  \psi_{\rm m}\iint_{\R \times \R}  (v(x) - v(y))^2 \rho_0(x)\rho_0(y)\,dxdy.
\end{aligned}
\end{align}
\end{lemma}
\begin{proof}From the definition of $\tilde\eta$ and $\tilde\omega$ we compute
\begin{align}\label{eq:rasys}
\begin{aligned}
\pa_t (\tilde \eta(x) - \tilde \eta(y) )& = (\tilde\omega(x) - \tilde\omega(y))  - \int_\R (\Psi(\eta(x)-\eta(z)) - \Psi(\eta_\infty(x)-\eta_\infty(z)))\rho_0(z)\,dz \cr
&\quad + \int_\R (\Psi(\eta(y)-\eta(z)) - \Psi(\eta_\infty(y)-\eta_\infty(z)))\rho_0(z)\,dz,\cr
\pa_t (\tilde \omega(x) - \tilde \omega(y)) & = - (\tilde \eta(x) - \tilde \eta(y)),
\end{aligned}
\end{align}
where we used $\pa_t \eta_\infty(x) = u_\infty$ for all $x \in {\rm supp}(\rho_0)$ by \eqref{asymp}. 

Testing \eqref{eq:rasys} against $(\tilde \eta(x)-\tilde \eta(y), \tilde \omega(x)-\tilde \omega(y))$ and integrating with respect to $\rho_0(x)\rho_0(y)$, we obtain
 \begin{align}\label{eq:free:1}
\begin{aligned}
&\frac12\frac{d}{dt}  \iint_{\R \times \R} \lt((\tilde \eta(x) - \tilde \eta(y))^2 + (\tilde\omega(x) - \tilde\omega(y))^2 \rt)\rho_0(x)\rho_0(y)\,dxdy\cr
&\quad = -\iiint_{\R \times \R \times \R} (\tilde \eta(x) - \tilde \eta(y))(\Psi(\eta(x)-\eta(z)) - \Psi(\eta_\infty(x)-\eta_\infty(z)))\rho_0(z)\rho_0(x)\rho_0(y)\,dxdydz \cr
&\qquad +\iiint_{\R \times \R \times \R} (\tilde \eta(x) - \tilde \eta(y)) (\Psi(\eta(y)-\eta(z)) - \Psi(\eta_\infty(y)-\eta_\infty(z)))\rho_0(z)\rho_0(x)\rho_0(y)\,dxdydz\cr
&\quad =2\iiint_{\R \times \R \times \R} (\tilde \eta(x) - \tilde \eta(y)) (\Psi(\eta(y)-\eta(z)) - \Psi(\eta_\infty(y)-\eta_\infty(z)))\rho_0(z)\rho_0(x)\rho_0(y)\,dxdydz\cr
&\quad = -2\iiint_{\R \times \R \times \R}  ( \tilde  \eta(x) - \tilde \eta(z)) (\Psi(\eta(y)-\eta(z)) - \Psi(\eta_\infty(y)-\eta_\infty(z)))\rho_0(z)\rho_0(x)\rho_0(y)\,dxdydz\cr
&\quad = - \iint_{\R \times \R} (\tilde \eta(y) - \tilde \eta(z)) (\Psi(\eta(y)-\eta(z)) - \Psi(\eta_\infty(y)-\eta_\infty(z)))\rho_0(z)\rho_0(y)\,dydz\cr
&\quad = - \iint_{\R \times \R}  \psi(\bar\zeta(x,y))(\tilde \eta(x) - \tilde \eta(y))^2 \rho_0(x)\rho_0(y)\,dxdy.
\end{aligned}
\end{align}
Here we used the symmetry of the integrals (exchanging $x$ and $y$ in the second term, then relabeling indices) and 
$$\begin{aligned}
\bar\zeta(x,y) = \bar\tau \lt(\eta(x)-\eta(y) \rt) + (1 - \bar\tau) (\eta_\infty(x)-\eta_\infty(y))
\end{aligned}$$
for some $\bar \tau \in [0,1]$. Since $\text{sgn}(\eta(x)-\eta(y)) = \text{sgn}(x-y)= \text{sgn}(\eta_\infty(x)-\eta_\infty(y))$, then they have the same sign, Proposition \ref{prop:1dlc:0} gives
\[
|\bar\zeta(x,y)| \leq \max\lt\{ D_{\eta}(t), 2\lt|\int_y^x \rho_0(z)\,dz\rt| \rt\} \leq \overline{D}.
\]
Hence $\psi(\bar\zeta(x,y))\ge \psi(\overline{D})=\psi_{\rm m}>0$, which proves \eqref{est_uw_l2_2}.  

Similarly, we also find
\begin{align}\label{est_uw_l2_3}
\begin{aligned}
&\frac12\frac{d}{dt}  \iint_{\R \times \R} \lt((\tilde \eta(x) - \tilde \eta(y))^2 + (v(x) - v(y))^2 \rt)\rho_0(x)\rho_0(y)\,dxdy\cr
&\quad = -\iint_{\R \times \R} \psi(\eta(x) - \eta(y))|v(x) - v(y)|^2 \,\rho_0(x) \rho_0(y)\,dxdy,
\end{aligned}
\end{align}
and again $\psi(\eta(x)-\eta(y))\ge\psi_{\rm m}$ by Proposition \ref{prop:1dlc:0}. This yields \eqref{est_uw_l2_4}.
\end{proof}

Now, for brevity of notation, we introduce
\begin{align*}
X &:=  \iint_{\R \times \R} (\tilde \eta(x) - \tilde \eta(y))^2  \rho_0(x)\rho_0(y)\,dxdy,\qquad
Y := \iint_{\R \times \R}  (\tilde\omega(x) - \tilde\omega(y))^2\rho_0(x)\rho_0(y)\,dxdy, \cr
Z &:= \iint_{\R \times \R}  (v(x) - v(y))^2\rho_0(x)\rho_0(y)\,dxdy,\qquad
\tilde{Z} :=   \iint_{\R \times \R}  \psi(\eta(x)-\eta(y))(v(x) - v(y))^2\rho_0(x)\rho_0(y)\,dxdy.
\end{align*}
With these definitions, inequalities \eqref{est_uw_l2_2} and \eqref{est_uw_l2_4} take the form
\bq\label{eq:abb}
\frac{d}{dt}(X+Y) \le -2\psi_{\rm m} X,\quad \frac{d}{dt}(X+Z) \le -2\psi_{\rm m} Z.
\eq
The estimates \eqref{eq:abb} provide dissipation in $X$ and $Z$ but not directly in $Y$. Nevertheless, they are already sufficient to deduce the decay of $X+Z$. Indeed, by \eqref{est_uw_l2_3} we have
\[
\frac{d}{dt}(X+Z) = - 2\tilde{Z},
\]
which shows that $\tilde{Z} \in L^1([0,\infty))$ since $X+Z$ is nonnegative. Hence, $(X+Z)(t)$ is absolutely continuous with finite limit as $t\to\infty$.
Analogously, we can see that $X,Z \in L^1([0,\infty))$ due to \eqref{eq:abb}. This yields
$(X + Z)(t) \to 0$ as $t \to \infty$, or equivalently,
\[
\iint_{\R \times \R} \lt((\tilde \eta(t,x) - \tilde \eta(t,y))^2 + (v(t,x) - v(t,y))^2 \rt)\rho_0(x)\rho_0(y)\,dxdy\to 0
\]
as $t \to \infty$.

This proves the $L^2$ energy decay for general nonnegative, non-decreasing, locally integrable communication weights $\psi$. To obtain quantitative convergence rates, however, we need to extract dissipation also in the $Y$-variables. This is the central step that distinguishes the bounded and the weakly singular weight cases, which will be analyzed in the following subsections.

%
%
%
%
%
%
%
%
%
%
\subsection{Bounded communication weights}

We first analyze the case where the communication weight $\psi$ is bounded. 
In this setting, we establish exponential convergence toward flocking states. 
The proof proceeds in two steps: we begin by deriving $L^2$ dissipation estimates for the perturbation variables, and then lift these estimates to uniform-in-$x$ ($L^\infty$) control of the trajectories and velocities. This two-step strategy will allow us to complete the proof of Proposition \ref{thm_2} (i).

%
%
%
%
%
%
%
%
%
%
\subsubsection{$L^2$ estimates}

To obtain dissipation in the $\tilde\omega$ variables, we recall \eqref{eq:tom} and estimate
\begin{align}\label{eq:omega:mani}
(\tilde\omega(x) - \tilde\omega(y))^2 
&\leq 3(v(x) - v(y))^2 + 3\int_\R \lt|\Psi(\eta(x)-\eta(z)) -\Psi(\eta_\infty(x)-\eta_\infty(z))\rt|^2 \rho_0(z)\,dz\cr
&\quad + 3\int_\R \lt|\Psi(\eta(y)-\eta(z)) -\Psi(\eta_\infty(y)-\eta_\infty(z))\rt|^2 \rho_0(z)\,dz.
\end{align}
By applying the mean value theorem and boundedness of $\psi$, we obtain
$$\begin{aligned}
&(\tilde\omega(x) - \tilde\omega(y))^2\cr
&\quad \leq 3(v(x) - v(y))^2  + 3\psi_{\rm M}^2\lt(\int_\R |\tilde\eta(x) - \tilde\eta(z)|^2\rho_0(z)\,dz + \int_\R |\tilde\eta(y) - \tilde\eta(z)|^2\rho_0(z)\,dz \rt).
\end{aligned}$$
Integrating with respect to $\rho_0(x)\rho_0(y)$, this yields
\begin{align*}
\begin{aligned}
&\iint_{\R \times \R} (\tilde\omega(x) - \tilde\omega(y))^2 \rho_0(x)\rho_0(y)\,dxdy\cr 
&\quad \leq  3\iint_{\R \times \R} (v(x) - v(y))^2\rho_0(x)\rho_0(y)\,dxdy  + 6\psi_{\rm M}^2\iint_{\R \times \R} |\tilde\eta(x) - \tilde\eta(y)|^2\rho_0(x)\rho_0(y)\,dxdy,
\end{aligned}
\end{align*}
this is,
$$\begin{aligned}
&-\iint_{\R \times \R} (v(x) - v(y))^2\rho_0(x)\rho_0(y)\,dxdy \cr
&\quad \leq  -\frac13\iint_{\R \times \R} (\tilde\omega(x) - \tilde\omega(y))^2 \rho_0(x)\rho_0(y)\,dxdy + 2\psi_{\rm M}^2\iint_{\R \times \R} |\tilde\eta(x) - \tilde\eta(y)|^2\rho_0(x)\rho_0(y)\,dxdy.
\end{aligned}$$
In terms of $X,Y, Z$ variables, this inequality can be rewritten as
\[
-Z \le -\frac{1}{3}Y + 2\psi_{\rm M}^2 X.
\]
We now combine this with the differential inequalities from Lemma \ref{lem:aux-energy}. Introducing a small parameter $\gamma>0$, we compute
\[
\begin{aligned}
\frac{d}{dt}((1+\gamma)X + Y + \gamma Z)  &= \frac{d}{dt}(X+Y) + \gamma\frac{d}{dt}(X+Z) \le -2\psi_{\rm m} X - 2\gamma \psi_{\rm m} Z \cr
 &\le -2\psi_{\rm m}X - \gamma \psi_{\rm m} Z - \frac{\gamma}{3}\psi_{\rm m} Y + 2\psi_{\rm M}^2 \gamma \psi_{\rm m} X \cr
 &\le  -\psi_{\rm m}\lt( (2 - 2\psi_{\rm M}^2\gamma) X + \frac{\gamma}{3}Y + \gamma Z \rt).
\end{aligned}
\]
Choosing 
\[
\gamma:= \frac{1}{1+2\psi_{\rm M}^2}<1,
\] 
we deduce
\[
\frac{d}{dt}((1+\gamma)X + Y + \gamma Z) \le -\frac{\psi_{\rm m} \gamma}{3}((1+\gamma)X + Y + \gamma Z).
\]
Thus, by Gr\"onwall's inequality, we have
\[
(1+\gamma)X(t) + Y(t) + \gamma Z(t) \le e^{-\frac{\psi_{\rm m} \gamma}{3}t}((1+\gamma)X_0 + Y_0 + \gamma Z_0),
\]
that is,
\bq\label{eq:L2exp}
\frac12  \iint_{\R \times \R} (\tilde \eta(x) - \tilde \eta(y))^2 + (v(x) - v(y))^2 +(\tilde\omega(x) - \tilde\omega(y))^2 \rho_0(x)\rho_0(y)\,dxdy \to 0
\eq
exponentially fast as $t \to \infty$. 

%
%
%
%
%
%
%
%
%
%
\subsubsection{$L^\infty$ estimates: proof of Proposition \ref{thm_2} (i)}

We now upgrade the $L^2$ decay to uniform-in-space convergence. The key idea is to introduce a modified energy functional containing a cross term between positions and velocities. This strategy allows the dissipation present in the velocity variable to be transferred to the position variable, and is a typical hypocoercivity-type argument. As a result, we obtain exponential convergence in the $L^\infty$ sense, complementing the $L^2$ estimates established above.

 For each $x \in {\rm supp}(\rho_0)$ we have
\begin{align*}
\pa_t (\eta(x) - \eta_\infty(x)) &= (v(x) - v_\infty), \\
\pa_t (v(x) - v_\infty(x)) &= - (\eta(x)-\eta_\infty(x)) -\int_\R \psi(\eta(x)-\eta(z))(v(x)-v(z))\rho_0(z)\,dz,
\end{align*}
where $v_\infty:=u_\infty$ as in \eqref{eq:uinfty:new}. 
Note that the interaction term can be decomposed as
\begin{equation*}
\begin{split}
& \int_\R \psi(\eta(x)-\eta(z))(v(x)-v(z))\rho_0(z)\,dz\\
&\quad = -\int_\R \psi(\eta(x)-\eta(z))((v(x)-v_\infty)-(v(z)-v_\infty))\rho_0(z)\,dz \\
&\quad =  -\lt(\int_\R \psi(\eta(x)-\eta(z))\rho_0(z)\,dz\rt) (v(x)-v_\infty) + \lt(\int_\R \psi(\eta(x)-\eta(z))(v(z)-v_\infty)\rho_0(z)\,dz\rt).
\end{split}
\end{equation*}
By Proposition \ref{prop:1dlc:0}, the first coefficient is bounded below by $\psi_{\rm m}>0$, i.e.
\[
\int_\R \psi(\eta(x)-\eta(z))\rho_0(z)\,dz \ge \psi(D_\eta(t)) \ge \psi(\overline{D}) \ge {\psi}_m>0.
\]
On the other hand, by recalling $
v_\infty = \int_\R \rho_0(y)v(y)\,dy$, the second term can be estimated in terms of the $L^2$ energy:
\bq\label{eq:help}
\begin{split}
\lt|\int_\R \psi(\eta(x)-\eta(z))(v(z)-v_\infty)\rho_0(z)\,dz\rt| &\le \psi_{\rm M} \lt|\int_\R |v(z)-v_\infty| \rho_0(z)\,dz\rt| \\
&\le \psi_{\rm M} \iint_{\R\times \R} |v(y)-v(z)|\rho_0(y)\rho_0(z)\,dydz \\
&\le \psi_{\rm M} \lt(\iint_{\R\times \R} |v(y)-v(z)|^2\rho_0(y)\rho_0(z)\,dydz\rt)^{1/2}\\
&=: \psi_{\rm M} g(t).
\end{split}
\eq
Here $g(t)$ decays exponentially fast as time goes to infinity thanks to \eqref{eq:L2exp}. With this decomposition, we obtain
$$
\begin{aligned}
\frac{1}{2}\frac{\pa}{\pa t} \lt((\eta(x)-\eta_\infty(x))^2 + (v(x)-v_\infty)^2\rt) &= -(v(x)-v_\infty) \lt(\int_\R \psi(\eta(x)-\eta(z))\rho_0(z)(v(x)-v(z))\,dz\rt)\cr
&\le- {\psi}_m(v(x)-v_\infty)^2 + |v(x)-v_\infty|\psi_{\rm M}g(t) \cr
&\le -\frac{{\psi}_m}{2}(v(x)-v_\infty)^2 + \frac{\psi_{\rm M}^2}{2{\psi}_m}g^2(t). 
\end{aligned}
$$
To close the estimate, we also consider the cross term
\[
\begin{split}
\frac{\pa}{\pa t} (\eta(x)-\eta_\infty)(v(x)-v_\infty) &=  (v(x)-v_\infty)^2 -(\eta(x)-\eta_\infty(x))^2  \\
 &\quad - (\eta(x)-\eta_\infty(x))\lt(\int_\R \psi(\eta(x)-\eta(z))(v(x)-v(z))\rho_0(z)\,dz\rt).
\end{split} 
\]
Using \eqref{eq:help}, we get
\[
\begin{split}
&\int_\R \psi(\eta(x)-\eta(z))(v(x)-v(z))\rho_0(z)\,dz\cr
 &\quad = \lt(\int_\R \psi(\eta(x)-\eta(z)) \rho_0(z)dz\rt) (v(x)-v_\infty)  - \int_\R \psi(\eta(x)-\eta(z)) (v(z)-v_\infty)\rho_0(z)dz \\
&\quad \le \psi_{\rm M} \lt (|v(x)-v_\infty| + g(t) \rt),
\end{split}
\]
and thus
\[
\frac{\pa}{\pa t} (\eta(x)- \eta_\infty)(v(x)-v_\infty) \le -\frac{1}{2}(\eta(x)-\eta_\infty(x))^2 + \lt(\frac{\psi_{\rm M}^2}{2}+1\rt) (v(x)-v_\infty)^2 + \frac{\psi_{\rm M}^2}{2}g^2(t).
\]
We now introduce the energy functional
\[
\mathcal{E}_\zeta(t) := \frac{1}{2}\lt((\eta(x)-\eta_\infty(x))^2 + (v(x)-v_\infty)^2\rt) + \zeta(\eta(x)-\eta_\infty(x))(v(x)-v_\infty),
\]
where $\zeta>0$ will be determined appropriately later. Then, combining the above estimates, we deduce
\[
\frac{\pa}{\pa t} \mathcal{E}_\zeta(t) \le  - \lt(\frac{{\psi}_m}{2}- \zeta \lt(\frac{\psi_{\rm M}^2}{2}+1\rt) \rt)(v(x)-v_\infty)^2 -\frac{\zeta}{2}(\eta(x)-\eta_\infty(x))^2 + \frac{\psi_{\rm M}^2}{2}\lt(\frac{1}{\psi_{\rm m}} + \zeta \rt)g^2(t).
\]
Taking
\[
\zeta := \min \lt\{\frac{{\psi}_m}{4}\lt(\frac{\psi_{\rm M}^2}{2}+1 \rt)^{-1}, \,\frac{1}{4}\rt\}, \quad c_\zeta:= \min \lt\{\frac{{\psi}_m}{4}, \frac{\zeta}{2}\rt\}, \quad C_\zeta:=  \frac{\psi_{\rm M}^2}{2}\lt(\frac{1}{{\psi}_m} + \zeta \rt),
\]
yields
\[
\begin{aligned}
\frac{\pa}{\pa t}\mathcal{E}_\zeta(t) \le - c_\zeta \lt((\eta(x)-\eta_\infty(x))^2 + (v(x)-v_\infty)^2\rt) + C_\zeta g^2(t)  \le - c_\zeta \mathcal{E}_\zeta(t) + C_\zeta g^2(t).
\end{aligned}
\]
By Gr\"onwall's inequality, we obtain
\[
\mathcal{E}_\zeta(t) \le e^{-c_\zeta t} \mathcal{E}_\zeta(0) + C_\zeta \int_0^t e^{-c_\zeta(t-s)} g^2(s)\,ds,
\]
and subsequently,
\[
\frac{1}{4} \lt((\eta(t,x)-\eta_\infty(t,x))^2 + (v(t,x)-v_\infty)^2\rt) \le  e^{-c_\zeta t} \lt( D_{\tilde{\eta}}^2(0) + D_{\tilde{v}}^2(0)\rt) + C_\zeta \int_0^t e^{-c_\zeta(t-s)} g^2(s)\,ds
\]
uniformly in $x \in {\rm supp}(\rho_0)$. Since $g(t)$ decays exponentially fast by \eqref{eq:help} and the $L^2$ estimate \eqref{eq:L2exp}, we conclude that
\[
 D_{\tilde \eta}(t) +  D_{\tilde v}(t) \to 0
\]
exponentially fast as $t \to \infty$.

Finally, recalling the definition of $\tilde\omega$ in \eqref{eq:tom}, the same exponential decay holds for $D_{\tilde\omega}(t)$.
This completes the proof of Proposition \ref{thm_2} (i).
%
%
%
%
%
%
%
%
%
%
\subsection{Weakly singular communication weights}

We now turn to the case of weakly singular  weights, when 
\[
\psi(r) \le \frac{B}{r^{\alpha}}, \quad \alpha \in (0,1).
\]
Here, the interaction weight is possibly unbounded near the origin, so the mean-value arguments used in the bounded case are no longer available. Instead, we derive an alternative estimate for the nonlinear term in \eqref{eq:omega:mani}, which will be the key tool for the decay analysis in this regime. 

The next lemma provides a bound that separates the effects of large and small relative 
displacements, and will serve as a central ingredient in what follows.

\begin{lemma}\label{lem:wsc}For any $\delta>0$ and any fixed $x\in{\rm supp}(\rho_0)$, we have
\[
\int_\R |\Psi(\eta(x)-\eta(z)) - \Psi(\eta_\infty(x) - \eta_\infty(z))|^2\rho_0(z)\,dz
\le  \lt( \frac{\Psi(\delta)}{\delta} \rt)^2 \int_\R |\tilde{\eta}(x) - \tilde{\eta}(z)|^2\rho_0(z)\,dz + (\Psi(\delta))^2\delta.
\]
\end{lemma}
\begin{proof}
The idea is to split the integration domain according to the size of the displacement. On the region where at least one of $|\eta(x)-\eta(z)|$ or $|\eta_\infty(x)-\eta_\infty(z)|$ exceeds $\delta$, Lemma \ref{lem:Psi:avg} provides a quadratic control in terms of $|\tilde{\eta}(x)-\tilde{\eta}(z)|^2$. On the complementary region where both differences are smaller than $\delta$, a direct bound by $\Psi(\delta)$ applies, and the measure of this set is controlled by $\delta$. We now make this argument precise.

Recall that $\Psi$ is odd, increasing and concave on $[0,\infty)$, and that
\[
{\rm sgn}(\eta(x)-\eta(z))={\rm sgn}(x-z)={\rm sgn}(\eta_\infty(x)-\eta_\infty(z)),
\]
thus the two arguments of $\Psi$ have the same sign (see \eqref{sgn} and $\pa_x\eta_\infty=2\rho_0\ge0$).

For a fixed $\delta > 0$, we decompose the integral domain into two parts:
\[
\begin{aligned}
&G_\delta:=\lt\{z: |\eta(x)-\eta(z)| \ge \delta\rt\} \cup \lt\{z: |\eta_\infty(x)-\eta_\infty(z)|\ge \delta \rt\},\cr
&B_\delta:= \lt\{z: |\eta(x)-\eta(z)| < \delta\rt\} \cap \lt\{z: |\eta_\infty(x)-\eta_\infty(z)| < \delta \rt\}.
\end{aligned}
\]

For $z\in G_\delta$, both arguments of $\Psi$ have the same sign and at least one of them has modulus greater than $\delta$.  By Lemma \ref{lem:Psi:avg} we have, for such $z$,
\[
  \int_{G_\delta}  |\Psi(\eta(x)-\eta(z)) - \Psi(\eta_\infty(x) - \eta_\infty(z))|^2\rho_0(z)\,dz \le \lt( \frac{\Psi(\delta)}{\delta} \rt)^2 \int_\R |\tilde{\eta}(x) - \tilde{\eta}(z)|^2\rho_0(z)\,dz.
\]
On the other hand, for $z \in B_\delta$, both differences are strictly less than $\delta$.  
Thus,
\[
 |\Psi(\eta(x)-\eta(z)) - \Psi(\eta_\infty(x) - \eta_\infty(z))| \le \Psi(\delta). 
\]
Moreover, as $B_\delta \subset \{z: |\eta_\infty(x)-\eta_\infty(z)| < \delta\}$ and
\[
 \{z: |\eta_\infty(x)-\eta_\infty(z)| < \delta\} = \lt\{z: 2 \lt|\int_x^z \rho_0(y)\,dy\rt| < \delta\rt\},
\] 
we obtain
\[
\int_{B_\delta} \rho_0(z)\,dz \le \delta.
\]
Combining these two, we deduce
\[
 \int_{B_\delta} |\Psi(\eta(x)-\eta(z)) - \Psi(\eta_\infty(x) - \eta_\infty(z))|^2\rho_0(z)\,dz \le (\Psi(\delta))^2\delta.
\]
Adding the estimates over $G_\delta$ and $B_\delta$ yields the desired bound.
\end{proof}

 By applying Lemma \ref{lem:wsc} to the nonlinear terms in \eqref{eq:omega:mani} and integrating \eqref{eq:omega:mani} with respect to $\rho_0(x)\rho_0(y)$, we get
$$\begin{aligned}
&-\iint_{\R \times \R} (v(x) - v(y))^2\rho_0(x)\rho_0(y)\,dxdy \cr
&\quad  \leq  -\frac13\iint_{\R \times \R} (\tilde\omega(x) - \tilde\omega(y))^2 \rho_0(x)\rho_0(y)\,dxdy\cr
&\qquad + 2\lt(\frac{\Psi(\delta)}{\delta}\rt)^2\iint_{\R \times \R} |\tilde\eta(x) - \tilde\eta(y)|^2\rho_0(x)\rho_0(y)\,dxdy + 2(\Psi(\delta))^2\delta.
\end{aligned}$$
In terms of the variables $X,Y,Z$, this reads
\[
-Z \le -\frac{1}{3}Y + 2\lt(\frac{\Psi(\delta)}{\delta}\rt)^2 X + 2(\Psi(\delta))^2\delta.
\]
Combining this bound with \eqref{eq:abb}, we obtain for any $\gamma>0$
\begin{align*}
\frac{d}{dt}((1+\gamma)X + Y + \gamma Z) \le  -\psi_{\rm m}\lt( \lt(2 - 2\lt(\frac{\Psi(\delta)}{\delta}\rt)^2 \gamma \rt) X + \frac{\gamma}{3}Y + \gamma Z \rt) + 2\gamma (\Psi(\delta))^2\delta.
\end{align*}

To proceed, we allow $\delta$ and $\gamma$ to be time-dependent decreasing profiles, chosen so as to control the error terms. In particular, we impose the constraint that both $\delta(t)$ and $\gamma(t)$ are decreasing and satisfy
\bq\label{gtdt:cons}
2\gamma(t) \lt(\frac{\Psi(\delta(t))}{\delta(t)}\rt)^2 \le 1.
\eq
Under this constraint, the differential inequality becomes
\[
\frac{d}{dt}((1+\gamma(t))X + Y + \gamma(t) Z) \le  -\frac{\psi_{\rm m} \gamma(t)}{3}\lt( (1+\gamma(t))X + Y + \gamma(t) Z \rt) + \delta^3(t).
\]

Since $\psi(r) \le B r^{-\alpha}$ for some $\alpha \in (0,1)$ and $B>0$, we deduce that
\[
\Psi(r) \le \frac{B}{1-\alpha} r^{1-\alpha} \quad \mbox{so that} \quad 2\lt(\frac{\Psi(r)}{r}\rt)^2 \le \frac{2B^2r^{-2\alpha}}{(1-\alpha)^2}.
\]
This suggests that to maximize the decay rate, $\gamma(t)$ should be chosen as large as possible under the constraint \eqref{gtdt:cons}. We thus set
\[
\gamma(t):=  \frac{(1-\alpha)^2\delta^{2\alpha}(t)}{2B^2}.
\]
With this choice, \eqref{gtdt:cons} holds automatically, and since $\delta(t)$ is taken decreasing, $\gamma(t)$ is also decreasing.

 To simplify the notations, we introduce
\bq\label{L:XYZ_2}
\calL(t):= (1+\gamma(t))X + Y + \gamma(t) Z,  
\eq
then we can rewrite the inequality above as
\bq\label{gold}
\dot{\calL}(t) \le - \kappa \delta^{2\alpha}(t)\calL(t) + \delta^3(t), \quad \kappa:= \frac{\psi_{\rm m}(1-\alpha)^2}{6B^2}>0.
\eq
Thus, the problem of establishing decay for $X,Y,Z$ reduces to studying the temporal behavior of $\calL(t)$. In particular, we are led to the following auxiliary lemma, which provides the desired algebraic decay rate.

\begin{lemma}\label{lem:delt:game}
Suppose that $\calL(t)$ satisfies the differential inequality \eqref{gold} with some $\kappa>0$ for all $t \in [0,\infty)$, for any positive decreasing profile $\delta(t)$. Then, we have
\[
\calL(t)= \mathcal{O} \lt(t^{-\lt(\frac{3}{2\alpha} -1\rt)}\rt), \quad t \to \infty.
\]
\end{lemma}
\begin{proof} To balance the dissipative term $-\kappa \delta^{2\alpha}(t)\calL(t)$ against the error $\delta^3(t)$, we choose a decreasing profile of the form
\[
\delta(t) = \frac{c_0}{(1+t)^\frac1{2\alpha}}, \mbox{
with $c_0$ satisfying }
c_0 > \frac{1}{\kappa^{\frac1{2\alpha}}} \lt(\frac3{2\alpha} - 1 \rt)^{\frac1{2\alpha}}.
\] 
Then, by Gr\"onwall's inequality, we have
\[
\lt(\calL(t) (1+t)^{\kappa c_0^{2\alpha}} \rt)' \leq c_0^3 (1+t)^{\kappa c_0^{2\alpha} - \frac3{2\alpha}}.
\]
This implies
\begin{align*}
\calL(t) &\leq \calL(0) (1+t)^{-\kappa c_0^{2\alpha}} + \frac{c_0^3}{\kappa c_0 +1 - \frac3{2\alpha}}\lt((1+t)^{1 - \frac3{2\alpha}} -(1+t)^{-\kappa c_0^{2\alpha}}\rt)\cr
&\leq \lt(\calL(0) + \frac{c_0^3}{\kappa c_0^{2\alpha} +1 - \frac3{2\alpha}}\rt) (1+t)^{-\lt(\frac3{2\alpha} - 1\rt)}.
\end{align*}
This completes the proof.
\end{proof}

We are now in a position to prove Proposition \ref{thm_2} (ii). The key idea is to combine the decay estimate for the Lyapunov functional $\calL(t)$ obtained in Lemma \ref{lem:delt:game} with an additional control on the velocity fluctuations $Z(t)$. This two-step approach allows us to capture the algebraic decay rates stated in the proposition, valid for all $\alpha \in (0,1)$.

\begin{proof}[Proof of Proposition \ref{thm_2} (ii)]
By Lemma \ref{lem:delt:game} and the definition \eqref{L:XYZ_2}, we obtain
\[
X(t) + Y(t)  \le \calL(t) = \mathcal{O} \lt(t^{-\lt(\frac{3}{2\alpha} -1\rt)}\rt).
\]
Moreover, choosing $\gamma(t)$ as in Lemma \ref{lem:delt:game} yields
\[
Z(t) \le \gamma^{-1}(t) \calL(t) = \mathcal{O} \lt(t^{-\lt(\frac{3}{2\alpha} -2\rt) }\rt).
\]
However, the exponent $\lt(\frac{3}{2\alpha} -2\rt) $ becomes negative for $\alpha \ge \frac{3}{4}$, thus this bound does not guarantee decay in that regime. To overcome this, we derive an alternative estimate for $Z(t)$.

Using a manipulation of \eqref{eq:tom} similar to \eqref{eq:omega:mani}, we find
$$\begin{aligned}
(v(x) - v(y))^2
&\leq 3(\tilde\omega(x) - \tilde\omega(y))^2  + 3\int_\R \lt|\Psi(\eta(x)-\eta(z)) -\Psi(\eta_\infty(x)-\eta_\infty(z))\rt|^2 \rho_0(z)\,dz\cr
&\quad + 3\int_\R \lt|\Psi(\eta(y)-\eta(z)) -\Psi(\eta_\infty(y)-\eta_\infty(z))\rt|^2 \rho_0(z)\,dz.
\end{aligned}$$
Since $|\Psi(a) -\Psi(b)| \le |\Psi(a-b)|$ whenever $a,b$ have same sign, we obtain
$$\begin{aligned}
(v(x) - v(y))^2
&\leq 3(\tilde\omega(x) - \tilde\omega(y))^2  + 3\int_\R \lt|\Psi(\tilde{\eta}(x)-\tilde{\eta}(z))\rt|^2 \rho_0(z)\,dz + 3\int_\R \lt|\Psi(\tilde{\eta}(y)-\tilde{\eta}(z))\rt|^2 \rho_0(z)\,dz.
\end{aligned}$$
By integrating both sides against $\rho_0(x)\rho_0(y)\,dxdy$ with $\psi(r)\le Br^{-\alpha}$ and H{\"o}lder's inequality, we deduce
$$\begin{aligned}
&\iint_{\R\times\R} (v(x)-v(y))^2\rho_0(x)\rho_0(y)\, dxdy \cr
&\quad \le 3\iint_{\R\times\R} (\tilde\omega(x) - \tilde\omega(y))^2\rho_0(x)\rho_0(y)\, dxdy + \frac{6B^2}{(1-\alpha)^2} \lt( \iint_{\R \times \R} |\tilde{\eta}(x)-\tilde{\eta}(y)|^2\rho_0(x)\rho_0(y)\,dxdy\rt)^{1-\alpha}\cr
&\quad \lesssim X(t)^{1-\alpha} +Y(t).
\end{aligned}$$
From the decay rate of $X(t)$ and $Y(t)$, we obtain
\[
\int_{\R\times\R} (v(x)-v(y))^2\rho_0(x)\rho_0(y)\, dxdy = \mathcal{O} \lt(t^{-\lt\{(1-\alpha)\lt(\frac{3}{2\alpha} -1\rt)\rt\} }\rt).
\]
Combining these two decay estimates of $Z(t)$, we obtain
\[
Z(t) =   \mathcal{O} \lt(t^{-  \max \lt\{ \lt(\frac{3}{2\alpha} -2\rt), (1-\alpha)\lt(\frac{3}{2\alpha} -1\rt) \rt\} }\rt) = \begin{cases}
\mathcal{O} \lt(t^{-\lt(\frac{3}{2\alpha} -2\rt) }\rt)  &\text{if }  \alpha \in (0,\frac{1}{2}], \\[2mm]
 \mathcal{O} \lt(t^{-\lt\{(1-\alpha)\lt(\frac{3}{2\alpha} -1\rt)\rt\} }\rt)   &\text{if } \alpha \in [\frac{1}{2},1).
\end{cases}
\]
This completes the proof.
\end{proof}

%
%
%
%
%

\subsection{Proof of Theorem \ref{thm_2:new}}
We first notice that due to Remark \ref{rmk:relax}, we are reduced to showing the result under the local assumption that $\psi$ holds for the $R=\overline{D}$ case. 

We start by identifying the asymptotic density profile in Eulerian coordinates.  
Recall from \eqref{asymp} that
\[
\eta_\infty(t,x) = \eta_c(t) + \eta_r(x), \quad \eta_r(x)=2\int_{-\infty}^x \rho_0(z)\,dz -1.
\]
This yields the following elementary observation.

\begin{lemma}\label{lem:rin:chr}
Let $\eta_\infty$ be as above. Then the pushforward of $\rho_0$ by $\eta_\infty(t,\cdot)$ is given by the uniform flock profile
\[
\rho_\infty(t,x):=\eta_\infty(t,\cdot)\#\rho_0
= \frac{1}{2}\,\mathbf{1}_{\{\eta_c(t)-1 \le x \le \eta_c(t)+1\}}.
\]
\end{lemma}
\begin{proof}
For any $\varphi \in C_b(\R)$, by definition of $\rho_\infty$, we have
\[
\int_{\R} \varphi(x)\,\rho_\infty(t,x)\,dx
= \int_{\R}\varphi(\eta_\infty(t,y))\,\rho_0(y)\,dy.
\]
Since $\eta_r$ is strictly increasing and $\eta_r'(y)=2\rho_0(y)$ a.e., we may change variables 
$s=\eta_r(y)\in[-1,1]$ with $ds=2\rho_0(y)\,dy$. Thus,
\[
\int_{\R}\varphi(\eta_\infty(t,y))\,\rho_0(y)\,dy
= \frac{1}{2}\int_{-1}^{1}\varphi(\eta_c(t)+s)\,ds
= \int_{\R}\varphi(x)\,\frac{1}{2}\,\mathbf{1}_{[\eta_c(t)-1,\,\eta_c(t)+1]}(x)\,dx,
\]
which establishes the claim.
\end{proof}

With Lemma \ref{lem:rin:chr} in hand, it follows that the asymptotic density profile $\rho_\infty$ coincides with the uniform distribution over an interval of length $2$, centered at the trajectory $\eta_c(t)$. Hence, in Eulerian coordinates, we have identified the expected limiting flock profile.

The passage from the Lagrangian decay estimates to the Eulerian framework proceeds exactly as in the proof of Theorem \ref{thm_1:new}. 
Indeed, $ {\rm d}_\infty(\rho(t),\rho_\infty(t))$, with $\rho_\infty(t)$ given by \eqref{asymprho}, can be bounded by $D_\eta(t)$ up to constants, and 
$\|u(t,\cdot)-u_\infty\|_{L^\infty(\text{supp}(\rho(t)))}$ is controlled by $D_v(t)$. 
Hence, the exponential (resp. algebraic) decay of diameters established in Section \ref{sec:m2} immediately yields the corresponding Eulerian stability estimates of Theorem \ref{thm_2:new}.



\section{Multidimensional dynamics with $(\lambda,\Lambda)$-convex interaction potential }\label{sec:m3}

We now extend our analysis to the multidimensional setting. 
In this case, the interaction potential $W$ is assumed to be $(\lambda,\Lambda)$-convex, and the communication weight $\psi$ admits a strictly positive lower bound. 
The convexity structure of $W$ provides direct coercivity estimates, which in turn enable us to derive uniform control of the Lagrangian flow map $\eta$ and its asymptotic convergence toward the equilibrium profile. This section aims to establish Proposition \ref{thm_3} below.
 
\begin{proposition}\label{thm_3}Let $d\geq1$ and let $(\eta,v)$ be a global solution to \eqref{main_sys2} with sufficient regularity. Suppose that the interaction potential $W$ satisfies \eqref{def:lam}, and that the communication weight $\psi$ has a strictly positive lower bound, i.e., there exists a $\psi_{\rm m} > 0$ such that $\psi(x) \geq \psi_{\rm m}$ for all $x \in \R^d$. Then the following decay estimates hold:
\begin{itemize}
\item[(i)]  If $\psi$ is bounded from above, then the $L^2$-type deviation in positions and velocities decays exponentially fast:
\[
\iint_{\R^d \times \R^d}\lt( |\eta(t,x) - \eta(t,y)|^2 + |v(t,x) - v(t,y)|^2\rt)\rho_0(x) \rho_0(y)\,dxdy \to 0 \quad \mbox{as } t \to \infty.
\]

\item[(ii)] If $\psi$ satisfies $\psi(x) \le B|x|^{-\alpha}$ for all $|x| \le R$ for some $\alpha \in (0,2)$ and $R>0$, the decay is algebraic:
\[
\iint_{\R^d \times \R^d}\lt( |\eta(t,x) - \eta(t,y)|^2 + |v(t,x) - v(t,y)|^2\rt)\rho_0(x) \rho_0(y)\,dxdy = \mathcal{O} \lt(t^{-\frac{2-\alpha}{\alpha}} \rt)
\]
as $t \to \infty$.
\end{itemize}
\end{proposition}

The starting point of our analysis is to derive a free energy identity for the Lagrangian system.
Using the same symmetrization arguments as in the one-dimensional case, we obtain from 
\eqref{main_sys2} the following dissipation relation:
\begin{align}\label{est_fe}
\begin{aligned}
&\frac{d}{dt}\iint_{\R^d \times \R^d} (W(\eta(x) - \eta(y)) -W(0)) \rho_0(x) \rho_0(y)\,dxdy  \cr
&\quad + \frac12\frac{d}{dt}\iint_{\R^d \times \R^d} |v(x) - v(y)|^2 \rho_0(x) \rho_0(y)\,dxdy\cr
&\qquad = - \iint_{\R^d \times \R^d} \psi(\eta(x) - \eta(y)) |v(x) - v(y)|^2 \rho_0(x) \rho_0(y)\,dxdy\cr
&\qquad \leq -\psi_{\rm m}\iint_{\R^d \times \R^d} |v(x) - v(y)|^2 \rho_0(x) \rho_0(y)\,dxdy,
\end{aligned}
\end{align}
where the last inequality follows from the assumption on $\psi(\cdot) \ge \psi_{\rm m}>0$. This identity represents the Lagrangian counterpart of the dissipation estimate \eqref{eq:diss} 
and will serve as the starting point of our analysis in higher dimensions.

In addition to this energy dissipation, we also need to control the cross term 
$\langle \eta(x)-\eta(y), v(x)-v(y)\rangle$. Differentiating this quantity in time 
and applying the symmetrization argument once again, we arrive at
\begin{align*}
&\frac{d}{dt}\iint_{\R^d \times \R^d} \lal \eta(x) - \eta(y), v(x) - v(y) \ral \rho_0(x) \rho_0(y)\,dxdy\cr
&\quad = \iint_{\R^d \times \R^d} |v(x) - v(y)|^2 \rho_0(x) \rho_0(y)\,dxdy\cr
&\qquad - \iiint_{\R^d \times \R^d \times \R^d} \lal \eta(x) -  \eta(y),  \nabla W(\eta(x) - \eta(z)) - \nabla W(\eta(y) - \eta(z))\ral \rho_0(x) \rho_0(y)\rho_0(z)\,dxdydz\cr
&\qquad - \iiint_{\R^d \times \R^d \times \R^d} \lal  \eta(x) - \eta(y), (\psi(\eta(x) - \eta(z))(v(x) - v(z)) \cr
&\hspace{5cm} - \psi(\eta(y) - \eta(z))(v(y) - v(z)) \ral\rho_0(x) \rho_0(y)\rho_0(z)\,dxdydz\cr
&\quad =: I + II + III.
\end{align*}
Using $\lambda$-convexity of $W$ in \eqref{def:lam},  $II$ can be estimated as 
$$\begin{aligned}
II  \leq -\lambda \iint_{\R^d \times \R^d} |\eta(x) - \eta(y)|^2 \rho_0(x) \rho_0(y)\,dxdy.
\end{aligned}$$
For the estimate of $III$, we use the symmetrization technique as in \eqref{eq:free:1} to get
$$\begin{aligned}
III &= 2 \iiint_{\R^d \times \R^d \times \R^d} \lal \eta(x) -  \eta(y), \psi(\eta(y) - \eta(z))(v(y) - v(z))\ral \rho_0(x) \rho_0(y)\rho_0(z)\,dxdydz \cr
&=-2\iiint_{\R^d \times \R^d \times \R^d} \lal  \eta(x) - \eta(z),  \psi(\eta(y) - \eta(z))(v(y) - v(z))\ral \rho_0(x) \rho_0(y)\rho_0(z) \,dxdydz\cr
&=-\iint_{\R^d \times \R^d} \lal \eta(y) -  \eta(z),  \psi(\eta(y) - \eta(z))(v(y) - v(z))\ral  \rho_0(y)\rho_0(z)\,dydz.
\end{aligned}$$
Putting these estimates together, we obtain
\begin{align}\label{est_m}
\begin{aligned}
&\frac{d}{dt}\iint_{\R^d \times \R^d} \lal  \eta(x) -  \eta(y), v(x) - v(y)\ral \rho_0(x) \rho_0(y)\,dxdy\cr
&\quad \leq \iint_{\R^d \times \R^d} |v(x) - v(y)|^2 \rho_0(x) \rho_0(y)\,dxdy -\lambda \iint_{\R^d \times \R^d} | \eta(x) -  \eta(y)|^2 \rho_0(x) \rho_0(y)\,dxdy\cr
&\qquad -\iint_{\R^d \times \R^d} \lal \eta(y) -  \eta(z), v(y) - v(z)\ral\psi(\eta(y) - \eta(z))  \rho_0(y)\rho_0(z)\,dydz.
\end{aligned}
\end{align}

Inequality \eqref{est_m} provides the fundamental balance between the $L^2$-type 
deviations of $\eta$ and $v$ in the multidimensional setting. The key step in deriving quantitative decay estimates is to control the last term on the right-hand side, which involves the communication weight $\psi$. If $\psi$ is uniformly bounded, this term can be controlled directly, leading to exponential decay. In contrast, if $\psi$ exhibits a singularity near the origin, a more delicate analysis is required, and only algebraic decay can be obtained.

From a structural perspective, the role of the mixed term 
$\lal \eta(x)-\eta(y), v(x)-v(y)\ral$ is essential: it transfers the dissipation 
from the velocity differences to the positional variables, thereby restoring coercivity 
at the level of a suitably chosen Lyapunov functional. This mechanism is a prototypical 
{\it hypocoercivity strategy}. In the bounded case, the resulting mixed functional yields 
exponential decay, while in the singular case, the same approach leads to algebraic rates.

%
%
%
%
%
\subsection{Bounded communication weights}
Let us assume that the communication weight function $\psi$ is bounded from above. 
In this case, the troublesome nonlinear term in \eqref{est_m} can be controlled by 
applying Cauchy--Schwarz and Young's inequalities, which yield
\begin{align*}
&-\iint_{\R^d \times \R^d} \lal  \eta(y) -  \eta(z), v(y) - v(z)\ral\psi(\eta(y) - \eta(z))  \rho_0(y)\rho_0(z)\,dydz\cr
&\quad \leq \frac\lambda 2  \iint_{\R^d \times \R^d} | \eta(x) -  \eta(y)|^2 \rho_0(x) \rho_0(y)\,dxdy  + \frac{\psi_{\rm M}^2}{2\lambda}\iint_{\R^d \times \R^d} |v(x) - v(y)|^2 \rho_0(x) \rho_0(y)\,dxdy.
\end{align*}
Combining this with \eqref{est_m} implies
\begin{align}\label{est_m1}
\begin{aligned}
&\frac{d}{dt}\iint_{\R^d \times \R^d} \lal  \eta(x) -  \eta(y), v(x) - v(y)\ral \rho_0(x) \rho_0(y)\,dxdy\cr
&\quad \le  -\frac\lambda 2 \iint_{\R^d \times \R^d} | \eta(x) -  \eta(y)|^2 \rho_0(x) \rho_0(y)\,dxdy\cr
&\qquad + \lt(1 + \frac{\psi_{\rm M}^2}{2\lambda} \rt) \iint_{\R^d \times \R^d} |v(x) - v(y)|^2 \rho_0(x) \rho_0(y)\,dxdy.
\end{aligned}
\end{align}

We then introduce a mixed Lyapunov functional by multiplying the above inequality by a small parameter $\zeta > 0$ and combining it with the dissipation estimate \eqref{est_fe}. Choosing $\zeta$ such that 
\[
\zeta < \min\lt\{ \frac{\lambda}{2}, \, \frac{1}{2}, \, \psi_{\rm m}\lt(1 + \frac{\psi_{\rm M}^2}{2\lambda}\rt)^{-1}\rt\},
\]
we define
$$\begin{aligned}
\mathcal{E}_\zeta(t)&:=\iint_{\R^d \times \R^d} (W(\eta(x) - \eta(y)) - W(0))\rho_0(x) \rho_0(y)\,dxdy\cr
&\qquad + \frac12\iint_{\R^d \times \R^d} |v(x) - v(y)|^2 \rho_0(x) \rho_0(y)\,dxdy\cr
&\qquad  + \zeta\iint_{\R^d \times \R^d} \lal  \eta(x) -  \eta(y), v(x) - v(y)\ral \rho_0(x) \rho_0(y)\,dxdy.
\end{aligned}$$
A direct computation using \eqref{est_fe} and \eqref{est_m1} shows 
$$\begin{aligned}
\frac{d}{dt}\mathcal{E}_\zeta(t)&\leq -\lt(\psi_{\rm m}- \zeta\lt(1 + \frac{\psi_{\rm M}^2}{2\lambda} \rt)\rt)\iint_{\R^d \times \R^d} |v(x) - v(y)|^2 \rho_0(x) \rho_0(y)\,dxdy\cr
& \quad -\frac{\lambda\zeta}{2} \iint_{\R^d \times \R^d} | \eta(x) -  \eta(y)|^2 \rho_0(x) \rho_0(y)\,dxdy,\cr
&\leq -\lt(2\psi_{\rm m}- \zeta\lt(2 + \frac{\psi_{\rm M}^2}{\lambda} \rt)\rt) \times \frac{1}{2}\iint_{\R^d \times \R^d} |v(x) - v(y)|^2 \rho_0(x) \rho_0(y)\,dxdy\cr
& \quad -\frac{\lambda\zeta}{\Lambda} \iint_{\R^d \times \R^d} (W(\eta(x)-\eta(y))-W(0)) \rho_0(x) \rho_0(y)\,dxdy,
\end{aligned}$$
due to \eqref{W_Lam}.  This gives
\bq\label{eq:md:gr}
\frac{d}{dt}\mathcal{E}_\zeta(t) \le  -C_{\lambda,\zeta} \mathcal{E}_\zeta(t),
\eq
where $C_{\lambda,\zeta} > 0$ is given by
\[
C_{\lambda,\zeta} := \min\lt\{\lt(2\psi_{\rm m}- \zeta\lt(2 + \frac{\psi_{\rm M}^2}{\lambda} \rt)\rt),  \frac{\lambda\zeta}{\Lambda} \rt\}.
\]
On the other hand, from \eqref{W_Lam} we have
\[
\mathcal{E}_\zeta(t) \geq \min\lt\{ \frac{\lambda}{2} -\zeta, \frac{1}{2}-\zeta \rt\}\iint_{\R^d \times \R^d} \lt( | \eta(x) - \eta(y)|^2 + |v(x) - v(y)|^2\rt) \rho_0(x) \rho_0(y)\,dxdy.
\]
Applying Gr\"onwall's lemma to \eqref{eq:md:gr} gives the exponential decay of $\mathcal{E}_\zeta(t)$, and subsequently, combined with the last inequality, we conclude that the $L^2$-type deviation of positions and velocities decays exponentially fast as $t\to\infty$. This proves Proposition \ref{thm_3} (i).

%
%
%
%
%
\subsection{Singular communication weights}
We now turn to the case of singular communication weights, where $\psi$ may blow up near the origin. Unlike the bounded case, the estimate \eqref{est_m} contains a nonlinear term that can no longer be handled by a direct Cauchy--Schwarz argument. Instead, we refine the strategy introduced in the previous subsection: rather than controlling this term by elementary inequalities, we exploit its exact structure and rewrite it in a potential form that naturally suggests the introduction of a new Lyapunov functional. This modified functional is specifically adapted to the singular setting and allows us to derive algebraic decay estimates.

Specifically, the last term on the right-hand side of \eqref{est_m} can be written as
\begin{align*}
&-\iint_{\R^d \times \R^d} \lal \eta(y) -  \eta(z), v(y) - v(z)\ral\psi(\eta(y) - \eta(z))  \rho_0(y)\rho_0(z)\,dydz\cr
&\quad = -\iint_{\R^d \times \R^d} \lt\lal \frac{\eta(y) -  \eta(z)}{|\eta(y) -  \eta(z)|}, v(y) - v(z)\rt\ral|\eta(y) -  \eta(z)|\psi(|\eta(y) - \eta(z)|)  \rho_0(y)\rho_0(z)\,dydz\cr
&\quad =-\frac{d}{dt}\iint_{\R^d \times \R^d} \lt( \int_0^{|\eta(x) - \eta(y)|} s\psi(s) \,ds \rt)\rho_0(x) \rho_0(y) \,dxdy.
\end{align*}
This identity shows that the singular factor can be expressed in terms of a new potential involving 
\[
\int_0^{|\eta(x)-\eta(y)|} s\psi(s)\,ds, 
\]
thereby motivating the introduction of a modified Lyapunov functional.

Combining this representation with \eqref{est_m}, we obtain
\begin{align}\label{est_rp}
\begin{aligned}
&\frac{d}{dt}\iint_{\R^d \times \R^d} \lal  \eta(x) -  \eta(y), v(x) - v(y)\ral \rho_0(x) \rho_0(y)\,dxdy \cr
&\quad + \frac{d}{dt}\iint_{\R^d \times \R^d} \lt( \int_0^{|\eta(x) - \eta(y)|} s\psi(s) \,ds \rt)\rho_0(x) \rho_0(y) \,dxdy \cr
&\qquad \leq \iint_{\R^d \times \R^d} |v(x) - v(y)|^2 \rho_0(x) \rho_0(y)\,dxdy -\lambda \iint_{\R^d \times \R^d} | \eta(x) -  \eta(y)|^2 \rho_0(x) \rho_0(y)\,dxdy.
\end{aligned}
\end{align}
This motivates us to define a modified energy functional, adapted to the singular setting:
\[
\begin{aligned}
\tilde{\mathcal{E}}_\xi(t) &:= \iint_{\R^d \times \R^d} (W(\eta(t,x) - \eta(t,y)) - W(0))\rho_0(x) \rho_0(y)\,dxdy  \cr
&\quad + \frac12\iint_{\R^d \times \R^d} |v(t,x) - v(t,y)|^2 \rho_0(x) \rho_0(y)\,dxdy \cr
&\quad + \xi\iint_{\R^d \times \R^d} \lal \eta(t,x) - \eta(t,y), v(t,x) - v(t,y)\ral \rho_0(x) \rho_0(y)\,dxdy\cr
&\quad + \xi\iint_{\R^d \times \R^d} \lt( \int_0^{|\eta(t,x) - \eta(t,y)|} s\psi(s) \,ds \rt)\rho_0(x) \rho_0(y) \,dxdy
\end{aligned}
\]
for some $\xi > 0$ satisfying $\xi < \min\{ \lambda, 1, \psi_{\rm m}\}$. The role of the additional terms is to absorb the singular interaction part in \eqref{est_rp} and ensure the correct coercivity of $\tilde{\mathcal{E}}_\xi(t)$ with respect to the $L^2$-type deviations.

Combining \eqref{est_fe} and \eqref{est_rp} gives
\begin{align*}
\begin{aligned}
\frac{d}{dt}\tilde{\mathcal{E}}_\xi(t) &\leq -(\psi_{\rm m}-\xi)\iint_{\R^d \times \R^d} |v(x) - v(y)|^2 \rho_0(x) \rho_0(y)\,dxdy   \cr
&\quad -\lambda\xi \iint_{\R^d \times \R^d} | \eta(x) -  \eta(y)|^2 \rho_0(x) \rho_0(y)\,dxdy,
\end{aligned}
\end{align*}
so that
\bq\label{eq:md:de}
\frac{d}{dt}\tilde{\mathcal{E}}_\xi(t) \leq -\min\lt\{ \psi_{\rm m} - \xi, \lambda \xi \rt\} \mathcal{L}(t),
\eq
where
\[
 \mathcal{L}(t):= \iint_{\R^d \times \R^d}\lt( |\eta(t,x) - \eta(t,y)|^2 + |v(t,x) - v(t,y)|^2\rt)\rho_0(x) \rho_0(y)\,dxdy.
\]
Using the $\lambda$-convexity of $W$ in \eqref{W_Lam} together with Young's inequality, we obtain
\begin{align*}
 &\iint_{\R^d \times \R^d} (W(\eta(x) - \eta(y)) - W(0))\rho_0(x) \rho_0(y)\,dxdy   + \frac12\iint_{\R^d \times \R^d} |v(x) - v(y)|^2 \rho_0(x) \rho_0(y)\,dxdy \cr
&\quad  + \xi \iint_{\R^d \times \R^d} \lal  \eta(x) -  \eta(y), v(x) - v(y)\ral \rho_0(x) \rho_0(y)\,dxdy\cr
&\qquad \geq \frac {(\lambda - \xi)}{2}\iint_{\R^d \times \R^d} | \eta(x) - \eta(y)|^2 \rho_0(x) \rho_0(y)\,dxdy \cr
&\qquad \quad+  \frac12\lt(1 - \xi\rt)\iint_{\R^d \times \R^d} |v(x) - v(y)|^2 \rho_0(x) \rho_0(y)\,dxdy.
\end{align*}
This shows that $\tilde{\mathcal{E}}_\xi(t)$ controls the $L^2$-type deviations:
$$\begin{aligned}
\tilde{\mathcal{E}}_\xi(t) &\geq \frac12\min\lt\{\lambda -\xi, 1 - \xi \rt\} \mathcal{L}(t) +  \xi\iint_{\R^d \times \R^d} \lt( \int_0^{|\eta(x) - \eta(y)|} s\psi(s) \,ds \rt)\rho_0(x) \rho_0(y) \,dxdy\cr
&\geq \frac12\min\lt\{\lambda -\xi, 1 - \xi \rt\} \mathcal{L}(t).
\end{aligned}$$
Together with \eqref{eq:md:de}, this yields a uniform-in-time bound
\bq\label{eq:calL:bd}
\mathcal{L}(t) \leq \lt(\frac{2}{\min\lt\{\lambda -\xi, 1 - \xi \rt\}} \rt)\tilde{\mathcal{E}}_\xi(t) \le  \lt(\frac{2}{\min\lt\{\lambda -\xi, 1 - \xi \rt\}} \rt)\tilde{\mathcal{E}}_\xi(0).
\eq

To derive a decay rate, we exploit the local behavior of $\psi$ near the origin. Recall that $\psi(s)=s^{-\alpha}$ with $\alpha \in (0,2)$ for $s \le R$. For any $X>0$, we have
\begin{align*}
\int_0^X s\psi(s)\,ds &= \int_0^{X\wedge R} s\psi(s)\,ds + \int_{X\wedge R}^X s\psi(s)\,ds \\
&\le \frac{(X \wedge R)^{2-\alpha}}{2-\alpha} + \frac{\psi(R)}{2}(X^2 - (X\wedge R)^2)\le \frac{X^{2-\alpha}}{2-\alpha} + \frac{\psi(R)}{2}X^2.
\end{align*}
As a consequence, the functional
\[
\mathcal{E}_\psi(t):= \iint_{\R^d \times \R^d} \lt(\int_0^{|\eta(t,x)-\eta(t,y)|} s\psi(s)\,ds\rt)\rho_0(dx)\rho_0(dy)
\]
satisfies
\[
\mathcal{E}_\psi(t)  \lesssim \iint_{\R^d \times \R^d} |\eta(t,x)-\eta(t,y)|^{2-\alpha}\rho_0(dx)\rho_0(dy) + \iint_{\R^d \times \R^d} |\eta(t,x)-\eta(t,y)|^{2}\rho_0(dx)\rho_0(dy).
\]
Applying H\"older's inequality to the first term and using the uniform bound \eqref{eq:calL:bd} for the second, we deduce
\[
\mathcal{E}_\psi(t) \lesssim \lt( \iint_{\R^d \times \R^d} |\eta(t,x)-\eta(t,y)|^{2}\rho_0(dx)\rho_0(dy)\rt)^{\frac{2-\alpha}{2}}.
\]

Combining this with \eqref{eq:calL:bd} and the definition of $\tilde{\mathcal{E}}_\xi$, we arrive at
\[
\tilde{\mathcal{E}}_\xi(t) \lesssim \mathcal{L}(t) + \mathcal{E}_\psi (t)\lesssim \mathcal{L}^{\frac{2-\alpha}{2}}(t), \quad \frac{2-\alpha}{2} \in (0,1).
\]
Thus, the differential inequality \eqref{eq:md:de} becomes
\[
\frac{d}{dt}\tilde{\mathcal{E}}_\xi(t) \leq  - C_0\lt(\tilde{\mathcal{E}}_\xi(t)\rt)^{\frac{2}{2-\alpha}},
\]
for some $C_0 > 0$ independent of $t$. Then we have
\[
\tilde{\mathcal{E}}_\xi(t) \le \lt(\tilde{\mathcal{E}}_\xi^{\frac{\alpha}{\alpha-2}}(0) + C_0\lt(\frac{\alpha}{2-\alpha}\rt)t\rt)^{-\frac{2-\alpha}{\alpha}},
\]
and subsequently, recalling the first inequality in \eqref{eq:calL:bd}, we obtain the algebraic decay as
\[
\calL(t) \lesssim \tilde{\mathcal{E}}_\xi(t) \lesssim t^{-\frac{2-\alpha}{\alpha}}
\]
for $t>0$ large enough. This completes the proof of the algebraic decay in Proposition \ref{thm_3} (ii).

%
%
%
%
%
%
%
%
%
%
%

\subsection{Proof of Theorem \ref{thm_3:new}}
Finally, we connect the multidimensional Lagrangian estimates of Proposition \ref{thm_3} with the Eulerian formulation. 
Since $\rho(t)$ is transported by the flow $\eta(t,\cdot)$, we apply Lemma \ref{lem:Wst} with $T_1=\eta(t,\cdot)$ and $T_2= \eta_c(t)$ to obtain 
\[
 {\rm d}_2(\rho(t),\delta_{\eta_c(t)})^2 \le \iint_{\R^d\times\R^d} |\eta(t,x)-\eta(t,y)|^2 \rho_0(x)\rho_0(y)\,dxdy,
\]
so that $ {\rm d}_2(\rho(t),\delta_{\eta_c(t)})$ is controlled precisely by the $L^2$ estimates in previous subsections.

For the velocity variables, we use again the conservation of total momentum:
\[
u_\infty = \int_{\R^d} u_0(x)\rho_0(x)\,dx = \int_{\R^d} u(\eta(t,y),t)\,\rho_0(y)\,dy \in \overline{\mathrm{conv}}\{u(t,x):x\in {\rm supp}\rho(t)\}.
\]
This implies
\[
\|u(t,\cdot)-u_\infty\|_{L^2(\R^d,d\rho(t))}^2 \le \iint_{\R^d\times\R^d} |u(t,x)-u(t,y)|^2 \rho(t,x)\rho(t,y)\,dxdy,
\]
so that the $L^2$-deviation of velocities is controlled by the dissipation functional in Proposition \ref{thm_3}.  

Therefore, Proposition \ref{thm_3} directly yields the Eulerian convergence estimates of Theorem \ref{thm_3:new}, with exponential decay in the case of bounded communication weights and algebraic decay in the presence of weakly singular weights.


\section*{Acknowledgments}
JAC was supported by the Advanced Grant Nonlocal-CPD (Nonlocal PDEs for Complex Particle Dynamics: Phase Transitions, Patterns and Synchronization) of the European Research Council Executive Agency (ERC) under the European Union Horizon 2020 research and innovation programme (grant agreement No. 883363) and partially supported by the EPSRC EP/V051121/1. JAC was partially supported by the ``Maria de Maeztu'' Excellence Unit IMAG, reference CEX2020-001105-M, funded by MCIN\slash AEI \slash10.13039\slash501100011033\slash.  YPC and DK were supported by the NRF grant no. 2022R1A2C1002820 and RS-202400406821. DK acknowledges support of the Institut Henri Poincar\'e (UAR 839 CNRS-Sorbonne Universit\'e), and LabEx CARMIN (ANR-10-LABX-59-01). OT acknowledges support by the Netherlands Organization for Scientific Research (NWO) under Grant No.\ 680.92.18.05.

%
%
%
%

\bibliographystyle{abbrv}
\bibliography{LT_hydro_final}

\end{document}